\pdfoutput=1
\RequirePackage{ifpdf}
\ifpdf 
\documentclass[pdftex]{sigma}
\else
\documentclass{sigma}
\fi

\usepackage[all]{xy}

\numberwithin{equation}{section}

\newtheorem{Theorem}{Theorem}[section]

\newtheorem{Lemma}[Theorem]{Lemma}
\newtheorem{Proposition}[Theorem]{Proposition}
 { \theoremstyle{definition}
\newtheorem{Definition}[Theorem]{Definition}

\newtheorem{Example}[Theorem]{Example}
\newtheorem{Remark}[Theorem]{Remark} }

 \DeclareMathOperator{\Tr}{Tr}
 \DeclareMathOperator{\tr}{tr}

\DeclareMathOperator{\spec}{spec}
\DeclareMathOperator{\AS}{AS}

\DeclareMathOperator{\APS}{APS}

\DeclareMathOperator{\Id}{Id}

\DeclareMathOperator{\ind}{index}

\DeclareMathOperator{\grad}{grad}
\DeclareMathOperator{\End}{End}

\DeclareMathOperator{\sgn}{sgn}

\DeclareMathOperator{\reg}{reg}

\DeclareMathOperator{\LIM}{LIM}
\DeclareMathOperator{\Imag}{Im}

\DeclareMathOperator{\Spin}{Spin}

 \DeclareMathOperator{\supp}{supp}

\DeclareMathOperator{\vol}{vol}

\newcommand{\Spinc}{\Spin^c}

 \newcommand{\R}{\mathbb{R}}
 \newcommand{\C}{\mathbb{C}}
 \newcommand{\N}{\mathbb{N}}
 \newcommand{\Z}{\mathbb{Z}}

\newcommand{\XX}{\mathfrak{X}}

\newcommand{\cF}{\mathcal{F}}

\begin{document}
\allowdisplaybreaks

\newcommand{\arXivNumber}{2110.00390}

\renewcommand{\PaperNumber}{023}

\FirstPageHeading

\ShortArticleName{Spectral Asymmetry and Index Theory on Manifolds with Generalised Hyperbolic Cusps}

\ArticleName{Spectral Asymmetry and Index Theory\\ on Manifolds with Generalised Hyperbolic Cusps}

\Author{Peter HOCHS~$^{\rm a}$ and Hang WANG~$^{\rm b}$}

\AuthorNameForHeading{P.~Hochs and H.~Wang}

\Address{$^{\rm a)}$~Institute for Mathematics, Astrophysics and Particle Physics, Radboud University,\\
\hphantom{$^{\rm a)}$}~PO Box 9010, 6500 GL Nijmegen, The Netherlands}
\EmailD{\href{mailto:p.hochs@math.ru.nl}{p.hochs@math.ru.nl}}

\Address{$^{\rm b)}$~School of Mathematical Sciences, East China Normal University,\\
\hphantom{$^{\rm b)}$}~No.~500, Dong Chuan Road, Shanghai 200241, P.R.~China}
\EmailD{\href{mailto:wanghang@math.ecnu.edu.cn}{wanghang@math.ecnu.edu.cn}}

\ArticleDates{Received June 22, 2022, in final form March 28, 2023; Published online April 20, 2023}

\Abstract{We consider a complete Riemannian manifold, which consists of a compact interior and one or more $\varphi$-cusps: infinitely long ends of a type that includes cylindrical ends and hyperbolic cusps. Here $\varphi$ is a function of the radial coordinate that describes the shape of such an end. Given an action by a compact Lie group on such a manifold, we obtain an equivariant index theorem for Dirac operators, under conditions on $\varphi$. These conditions hold in the cases of cylindrical ends and hyperbolic cusps. In the case of cylindrical ends, the cusp contribution equals the delocalised $\eta$-invariant, and the index theorem reduces to Donnelly's equivariant index theory on compact manifolds with boundary. In general, we find that the cusp contribution is zero if the spectrum of the relevant Dirac operator on a~hypersurface is symmetric around zero.}


\Classification{58J20; 58D19}

\section{Introduction}

\subsection*{Index theory on noncompact manifolds}

Index theory on noncompact, complete manifolds comes up naturally in several different contexts. Well-known results include:
\begin{enumerate}\itemsep=0pt
\item[(1)] the Gromov--Lawson relative index theorem \cite{Gromov83} for differences of Dirac operators that are invertible, and equal, outside compact sets;
\item[(2)] the Atiyah--Patodi--Singer (APS) index theorem \cite{APS1, Donnelly} on compact manifolds with boundary, where an approach is to attach a cylindrical end to the boundary to obtain a complete manifold without boundary; and
\item[(3)] index theorems on noncompact locally symmetric spaces, and manifolds with cusps modelled on such spaces. Some important results include \cite{BM, Moscovici82, Muller83, Muller87, Stern89}.
\end{enumerate}
Other important areas, which are not considered in this paper, are
index theory of Callias-type operators \cite{Anghel89, Anghel93, Bott78, Bunke95, Callias78, Kucerovsky01}, index theory where a group action is used to define an equivariant index of operators that are not Fredholm in the traditional sense, see, e.g., \cite{Atiyah74, Braverman02}, and index theory with values in the $K$-theory of $C^*$-algebras, see, e.g., \cite{Connes94}.

In this paper and in~\cite{HW21a}, we work towards a common framework for studying the three types of index problems mentioned. In \cite{HW21a}, we considered a complete Riemannian manifold $M$, a~Clifford module $S \to M$ and a Dirac operator $D$ on $\Gamma^{\infty}(S)$ that is ``invertible at infinity'' in the following sense. We assumed that there are a compact subset $Z \subset M$ and a $b>0$ such that for all $s \in \Gamma_c^{\infty}(S)$ supported outside $Z$,
\[
\|Ds\|_{L^2} \geq b \|s\|_{L^2}.
\]
Then $D$ is Fredholm as an unbounded operator on $L^2(S)$ with a suitable domain \cite{Anghel93b, Gromov83}.

We assumed that $M$ has a warped product structure outside $Z$ (the $\varphi$-cusps as below, without assumptions on the function $\varphi$).
Furthermore, we considered an action by a compact group $G$ on $M$ and $S$, commuting with $D$, and a group element $g \in G$. The main result in \cite{HW21a} is an expression for the value at $g$ of the equivariant index of such an operator, as an Atiyah--Segal--Singer-type contribution from $Z$ and a contribution from outside $Z$. This implies an equivariant version of the second index theorem mentioned at the start, and an equivariant version of the first for manifolds with the appropriate warped product form at infinity.

In this paper, we give an expression for the contribution from outside $Z$ for manifolds with specified shapes outside $Z$, including cylindrical ends and hyperbolic cusps.

\subsection*{$\boldsymbol\varphi$-cusps}

More specifically, let $M$ be a complete Riemannian manifold. Suppose a compact Lie group $G$ acts isometrically on $M$, that $S \to M$ is a $G$-equivariant Clifford module, and that $D$ is a $G$-equivariant Dirac operator on sections of $S$.
Let $a>0$, and let $\varphi \in C^{\infty}(a, \infty)$. We assume that there is a compact, $G$-invariant subset $Z \subset M$ with smooth boundary $N$, such that $C := M \setminus Z$ is $G$-equivariantly isometric to the product $N \times (a, \infty)$, with the Riemannian metric
\begin{equation*} 
{\rm e}^{2\varphi}\big(B_N + {\rm d}x^2\big),
\end{equation*}
where $B_N$ is a $G$-invariant Riemannian metric on $N$, and $x$ is the coordinate in $(a, \infty)$. Then we say that $M$ has \emph{$\varphi$-cusps}. (The results in this paper extend to cases where different functions~$\varphi$ are used on different connected components of $N$.)

A natural form of a Dirac operator on $C$ is
\[
{\rm e}^{-\varphi} c_0\bigg(\frac{\partial}{\partial x} \bigg)\bigg(\frac{\partial}{\partial x}+D_N +\frac{\dim(M)-1}{2}\varphi' \bigg)
\]
for a Dirac operator $D_N$ on $S|_N$, where $c_0$ is the Clifford action for the product metric
$B_N + {\rm d}x^2$ on $C$. We assume $D|_C$ has this form, and, initially, that $D_N$ is invertible. Then $D_N^2 \geq b^2$ for some $b>0$.

We say that $M$ has
\emph{weakly admissible $\varphi$-cusps} if
\begin{enumerate}\itemsep=0pt
\item[(1)] $\varphi$ is bounded above, and
\item[(2)] there is an $\alpha>0$ such that $|\varphi'(x)| \leq b - \alpha$ for large enough $x$,
\end{enumerate}
and
\emph{strongly admissible $\varphi$-cusps} if
\begin{enumerate}\itemsep=0pt
\item[(1)] $\lim_{x \to \infty} \varphi(x) =-\infty$, and
\item[(2)] $\lim_{x \to \infty} \varphi'(x) =0$.
\end{enumerate}
For example, if $\varphi(x) = -\mu \log(x)$ for $\mu \in \R$, then
\begin{itemize}\itemsep=0pt
\item $M$ is complete if and only if $\mu \leq 1$,
\item $M$ has weakly admissible $\varphi$-cusps if and only if $\mu \geq 0$, and
\item $M$ has strongly admissible $\varphi$-cusps if and only if $\mu > 0$.
\end{itemize}
Furthermore, $M$ has finite volume if and only if $\mu>1/\dim(M)$, but this is not directly relevant to us here.

\subsection*{Results}

If $M$ has weakly admissible $\varphi$-cusps, then our main result, Theorem~\ref{thm index cusps} states that $D$ is Fredholm, and the value of its equivariant index at $g \in G$ is
\begin{equation} \label{eq index intro}
\ind_G(D)(g) = \int_{Z^g} \AS_g(D) - \frac{1}{2}\eta_g^{\varphi}\big(D_N^+\big).
\end{equation}
Here
\begin{itemize}\itemsep=0pt
\item $Z^g$ is the fixed point set of $g$ in $Z$,
\item $\AS_g(D)$ is the Atiyah--Segal--Singer integrand for $D$, and
\item $\eta_g^{\varphi}\big(D_N^+\big)$ is the \emph{$\varphi$-cusp contribution} associated to $D_N^+$ (the restriction of $D_N$ to even-graded sections).
\end{itemize}
The cusp contribution $\eta_g^{\varphi}\big(D_N^+\big)$ equals
\begin{equation} \label{eq cusp contr intro}
\eta^{\varphi}_g\big(D_N^+\big) =
\lim_{a'\downarrow a}
 \int_{0}^{\infty}
\sum_{\lambda \in \spec(D_N^+)}\sgn(\lambda) \tr(g|_{\ker(D_N^+ - \lambda)})
F_{\varphi}(a', s, |\lambda|) \, {\rm d}s
\end{equation}
for a function $F_{\varphi}$ depending on $\varphi$. (Here $a$ is as in the definition of $C \cong N \times (a, \infty)$.) This function is expressed
 in terms of eigenfunctions of a Sturm--Liouville (or Schr\"odinger) operator on the half-line $(0, \infty)$, with Dirichlet boundary conditions at $0$. See Definition~\ref{def cusp cpt} for details. If $M$ has strongly admissible $\varphi$-cusps, then this operator has discrete spectrum. If $M$ only has weakly admissible $\varphi$-cusps, then it may have a continuous spectral decomposition, and its spectral measure also appears in the expression for $\eta_g^{\varphi}\big(D_N^+\big)$.

There is also a version of~\eqref{eq index intro} where the integral over $Z^g$ is replaced by an integral over $M^g$, if this converges, and the limit in~\eqref{eq cusp contr intro} is replaced by the limit $a' \to \infty$.

A possibly interesting feature of the cusp contribution $\eta_g^{\varphi}\big(D_N^+\big)$ is that it equals zero if the spectrum of $D_N^+$
has the equivariant symmetry property that
\begin{equation} \label{eq g symm spec intro}
\tr\big(g|_{\ker(D_N^+ - \lambda)}\big) = \tr\big(g|_{\ker(D_N^+ + \lambda)}\big)
\end{equation}
for all $\lambda \in \R$. If $g=e$, then this is exactly symmetry of the spectrum with respect to reflection in $0$, including multiplicities. It is immediate from~\eqref{eq cusp contr intro} that $\eta^{\varphi}_g\big(D_N^+\big)=0$ if the spectrum of $D_N^+$ has this property.
 As noted in~\cite{APS1}, the classical $\eta$-invariant also vanishes if $D_N^+$ has symmetric spectrum. So it seems that different ways of measuring spectral asymmetry are relevant to index theory on manifolds of the type we consider.

After we prove~\eqref{eq index intro}, we compute the function $F_{\varphi}$ in~\eqref{eq cusp contr intro}, and hence the cusp contribution $\eta_g^{\varphi}\big(D_N^+\big)$ in the case where $\varphi(x) = 0$. Then $C = N \times (a, \infty)$ is a cylindrical end, and Proposition~\ref{prop cusp contr cylinder} states that $\eta_g^{0}\big(D_N^+\big)$ is the equivariant $\eta$-invariant \cite{Donnelly} of~$D_N^+$. This computation is a~spectral version of the geometric computation in \cite[Section~5]{HW21a}. Then~\eqref{eq index intro} becomes Donnelly's equivariant version of the APS index theorem~\cite{Donnelly}.

The eigenfunctions of the Sturm--Liouville operator involved in the expression~\eqref{eq cusp contr intro} are known explicitly in several cases besides the cylinder case, such as $\varphi(x) = -\log(x)/2$ and the hyperbolic cusp case $\varphi(x) = -\log(x)$. Nevertheless,
 it seems to be a nontrivial problem to evaluate the cusp contribution~\eqref{eq cusp contr intro} explicitly, even when these eigenfunctions are known. Concrete consequences and special cases of the main Theorem~\ref{thm index cusps} are:
\begin{enumerate}\itemsep=0pt\samepage
\item[(1)] the Fredholm property of $D$,
\item[(2)] the fact that, for $\varphi = 0$, the cusp contribution $\eta_g^0\big(D_N^+\big)$ is the delocalised $\eta$-invariant of~$D_N^+$,
\item[(3)] the fact that, for general $\varphi$, the cusp contribution vanishes if the spectrum of $D_N^+$ has the symmetry property~\eqref{eq g symm spec intro}.
\end{enumerate}

Related index theorems were obtained for manifolds with ends of the form $N \times (a, \infty)$ with metrics of the form $B_{N, x} + {\rm d}x^2$, where now $B_{N, x}$ is a Riemannian metric on $N$ depending on $x\in (a, \infty)$. In many cases, this family of metrics on $N$ has the form $B_{N, x} = \rho^2(x)B_N$, for a fixed Riemannian metric $B_N$ on $N$ and a function $\rho$ on $(a, \infty)$. Results in this context include the ones in
 \cite{Baier02, BB03, BBC12, Stern84, Vaillant01}. We discuss the relations between these results and~\eqref{eq index intro} in Section~\ref{sec other results}.

\section{Preliminaries and result} \label{sec results}

Throughout this paper, $M$ is a $p$-dimensional Riemannian manifold with $p$ even, and $S = S^+ \oplus S^-\to M$ is a $\Z/2$-graded Hermitian vector bundle.
We denote the Riemannian density on $M$ by ${\rm d}m$. We also assume that a compact Lie group $G$ acts smoothly and isometrically on~$M$, that $S$ is a $G$-equivariant vector bundle and that the action on $G$ preserves the metric and grading on $S$. We fix, once and for all, an element $g \in G$.

\subsection[varphi-cusps]{$\boldsymbol{\varphi}$-cusps}

\begin{Definition}\label{def cusps}
The manifold $M$ has ($G$-invariant) \emph{$\varphi$-cusps} if there are
 \begin{itemize}\itemsep=0pt
 \item a $G$-invariant compact subset $Z \subset M$ with smooth boundary $N$, and
 \item a number $a\geq 0$ and a function $\varphi \in C^{\infty}(a,\infty)$,
 \end{itemize}
 such that
 \begin{itemize}\itemsep=0pt
 \item
 there is a $G$-equivariant isometry from
$C:= {M \setminus Z}$ onto the manifold $N \times (a,\infty)$ with the metric
\begin{equation} \label{eq g phi}
 B_{\varphi} := {\rm e}^{2\varphi}\big(B_N + {\rm d}x^2\big),
\end{equation}
where $B_N$ is the restriction of the Riemannian metric to $N$ and $x$ is the coordinate in $(a,\infty)$, and
\item
this isometry has a continuous extension to a map $\overline{C} \to N \times [a, \infty)$, which maps $N$ onto $N \times \{a\}$.
\end{itemize}
The manifold $M$ has \emph{strongly admissible $\varphi$-cusps} or \emph{strongly admissible cusps} if, in addition,
\begin{equation} \label{eq strong adm}
\begin{split}
\lim_{x \to \infty} \varphi(x) &= -\infty,\qquad
\lim_{x \to \infty} \varphi'(x) &= 0.
\end{split}
\end{equation}
\end{Definition}

\begin{Remark}
There is no loss of generality in assuming that $a=0$ in Definition~\ref{def cusps}. However, in examples, it may be that $\varphi$ arises as the restriction to $(a, \infty)$ of a function naturally defined on say $(0, \infty)$. Allowing nonzero $a$ then means that we do not need to shift these functions over~$a$ to obtain a function on~$(0, \infty)$. This does not matter for the results.
\end{Remark}

\begin{Remark}In Definition~\ref{def cusps}, the hypersurface $N$ may be disconnected. Let $N_1, \dots, N_k$ be its connected components. All results below generalise to the case where the Riemannian metric on $N \times (a,\infty)$ is of the form ${\rm e}^{2\varphi_j}(B_N + {\rm d}x^2)$ on $N_j \times (a, \infty)$, for $\varphi_j \in C^{\infty}(a,\infty)$ depending on~$j$. This generalisation is straightforward, and we do not work out details here.
\end{Remark}

\begin{Lemma}\label{lem phi complete}
If $M$ has $\varphi$-cusps, then it is complete if and only if
\begin{equation} \label{eq phi complete}
\int_{a+1}^{\infty} {\rm e}^{\varphi(x)}\, {\rm d}x = \infty.
\end{equation}
\end{Lemma}
\begin{proof}
For any smooth, increasing $\gamma\colon [0,1) \to [a+1,\infty)$ such that $\gamma(0) = a+1$, and any $n \in N$, consider the curve $\gamma_{n}$ in $C$ given by $\gamma_{n}(t) = (n, \gamma(t))$. Then the length of $\gamma_{n}$ is
\[
\sgn(\gamma') \int_0^1 \gamma'(s) {\rm e}^{\varphi(\gamma(t))}\, {\rm d}t = \sgn(\gamma') \int_{a+1}^{\gamma(1)} {\rm e}^{\varphi(x)}\, {\rm d}x.
\]
Here $\gamma(1):= \lim_{ t\to 1}\gamma(t)$ exists because $\gamma$ is increasing. The lengths of all such curves with $\gamma(1)= \infty$ are infinite if and only if~\eqref{eq phi complete} holds. This implies that~\eqref{eq phi complete} is equivalent to the condition that any curve in $C$ that goes to infinity has infinite length.
Compare also the proof of Theorem 1 in \cite{Nomizu61}.
\end{proof}

From now on, suppose that $M$ is complete and has {$\varphi$-cusps}, and let $Z$, $N$, $a$, $\varphi$ and $B_N$ be as in Definition~\ref{def cusps}.

It will not directly be important for our results if $M$ has finite or infinite volume. But because of the relevance of finite-volume manifolds, we note the following fact. Recall that $p$ is the dimension of $M$.
\begin{Lemma}\label{lem fin vol}
The manifold $M$ has finite volume if and only if
\[
\int_a^{\infty} {\rm e}^{p\varphi(x)}\, {\rm d}x < \infty.
\]
\end{Lemma}
\begin{proof}
The Riemannian density for $B_{\varphi}$ is
$
\vol_{B_{\varphi}} = {\rm e}^{p \varphi} \vol_{B_N} \otimes {\rm d}x.
$
\end{proof}

Lemmas~\ref{lem phi complete} and~\ref{lem fin vol} are well known and our results do not logically depend on them. We chose to include their short proofs for the sake of completeness and to illustrate properties of $\varphi$-cusps.

\begin{Example}\label{ex mu log x}
Suppose that $M$ has {$\varphi$-cusps}, with $\varphi(x) = -\mu \log(x)$ for $\mu \in \R$. Then
\begin{itemize}\itemsep=0pt
\item
$M$ has strongly admissible {cusps} if and only if $\mu>0$,
\item $M$ is complete if and only if $\mu \leq 1$,
\item $M$ has finite volume if and only if $\mu > 1/p$.
\end{itemize}
All three conditions hold in the case $\mu = 1$ of hyperbolic cusps. In the case $\mu = 0$ of a cylindrical end, $M$ has infinite volume and does not have strongly admissible cusps. But then $M$ has weakly admissible cusps as in Definition~\ref{def weakly adm} below, which is sufficient for our purposes.
\end{Example}

\subsection[Dirac operators on varphi-cusps]{Dirac operators on $\boldsymbol{\varphi}$-cusps}

From now on, we suppose that $S$ is a $G$-equivariant Clifford module, which means that there is a $G$-equivariant vector bundle homomorphism $c\colon TM \to \End(S)$, with values in the odd-graded endomorphisms, such that for all $v \in TM$,
\[
c(v)^2 = -\|v\|^2 \Id_S.
\]
Let $\nabla$ be a $G$-invariant connection on $S$ that preserves the grading. Suppose that for all vector fields $v$ and $w$ on $M$,
\begin{equation} \label{eq Clifford conn}
[\nabla_v, c(w)] = c\big(\nabla^{TM}_v w\big),
\end{equation}
where $\nabla^{TM}$ is the Levi-Civita connection. Consider the Dirac operator
\begin{equation} \label{eq Dirac}
D\colon\ \Gamma^{\infty}(S) \xrightarrow{\nabla} \Gamma^{\infty}(S \otimes T^*M) \cong
 \Gamma^{\infty}(S \otimes TM) \xrightarrow{c} \Gamma^{\infty}(S).
\end{equation}
It is odd with respect to the grading on $S$; we denote its restrictions to even- and odd-graded sections by $D^{\pm}$, respectively.

Suppose that we have a $G$-equivariant vector bundle isomorphism $S|_{C} \cong S|_N \times (a,\infty) \to N \times (a, \infty)$.

We will assume that $D$ has a natural form on $M\setminus Z$,~\eqref{eq DC Dphi} below.
This assumption is motivated by a special case, Proposition~\ref{prop Dphi}, which we discuss now.

Let $B_0 := B_N + {\rm d}x^2$ be the product metric on $N \times (a,\infty)$.
Then
\[
c_0:= {\rm e}^{-\varphi} c|_{C}\colon\ T(N \times (a,\infty)) \to \End(S|_{C})
\]
is a Clifford action for the metric $B_0$. Let $\{e_1,\dots, e_p\}$ be a local orthonormal frame for $TM$ with respect to $B_0$, with $e_p = \frac{\partial}{\partial x}$, and $\{e_1,\dots, e_{p-1}\}$ a local orthonormal frame for $TN$ with respect to $B_N$. (The objects that follow are defined globally by their expressions in terms of this frame, because they do not depend on the orthonormal frame.)

Because $p$ is even, the operator
\[
\gamma := (-{\rm i})^{p(p+1)/2} c_0(e_1) \cdots c_0(e_p)
\]
defines a $\Z/2$-grading
on $S$, and $c_0(v)$ is odd for this grading for all vector fields $v$ on $C$. We suppose that the given grading on $S$ equals $\gamma$ on $S|_C$.
Let $\nabla^0$ be a Clifford connection on $S|_C$ with respect to $B_0$ and $c_0$ (i.e., a Hermitian connection satisfying~\eqref{eq Clifford conn} for $B_0$ and $c_0$), and suppose that it preserves the grading $\gamma$, and that $\nabla_{\frac{\partial}{\partial x}} = \frac{\partial}{\partial x}$.

Let
\[
D_N := -c_0(e_p) \sum_{j=1}^{p-1}c_0(e_j) \nabla^0_{e_j}.
\]
This is a Dirac operator on $S|_N$ with respect to the Clifford multiplication
\[
c_N(v) := -c_0(e_p)c_0(v)
\]
for $v \in TN$.

\begin{Proposition}\label{prop Dphi}
There is a Clifford connection on $S|_C$, with respect to the Clifford action $c$ and the metric~\eqref{eq g phi},
 such that the resulting Dirac operator is
\begin{equation} \label{eq Dphi}
{\rm e}^{-\varphi} c_0(e_p)\bigg(\frac{\partial}{\partial x}+D_N +\frac{p-1}{2}\varphi' \bigg).
\end{equation}
\end{Proposition}

This fact follows from a standard expression for conformal transformations of Dirac operators. (See, e.g., the proof of Proposition 1.3 in \cite{Hitchin74} in the $\Spinc$ case.) We summarise the arguments in Appendix~\ref{app Dirac} for the sake of completeness.

From now on, we make the two assumptions that
\begin{enumerate}\itemsep=0pt
\item[(1)] for some Dirac operator $D_N$ on $S|_N$ that preserves the grading, and some grading-reversing, $G$-equivariant, isometric vector bundle endomorphism $\sigma\colon S|_N \to S|_N$ that anti-commutes with $D_N$,
\begin{equation} \label{eq DC Dphi}
 D|_{M \setminus Z} = {\rm e}^{-\varphi} \sigma \bigg(\frac{\partial}{\partial x}+D_N +\frac{p-1}{2}\varphi' \bigg),
 \end{equation}
 \item[(2)] the Dirac operator $D_N$ is invertible.
\end{enumerate}
The first assumption is satisfied for a natural choice of Clifford connection on $S$, by Proposition~\ref{prop Dphi}. We indicate how to remove the second assumption in Section~\ref{sec DN not invtble}.

Because $D_N$ is invertible, there is a $b>0$ such that $D_N^2 \geq b^2$. For the purposes of our index theorem, the strong admissibility condition in Definition~\ref{def cusps} may be weakened to the following.
\begin{Definition}\label{def weakly adm}
The manifold $M$ has \emph{weakly admissible $\varphi$-cusps} or \emph{weakly admissible cusps} (with respect to $D_N$) if $\varphi$ is bounded above, and there are $a'\geq a$ and $\alpha>0$ such that for all $x \in (a', \infty)$,
\begin{equation} \label{eq bound phi' b}
|\varphi'(x)| \leq b - \alpha.
\end{equation}
\end{Definition}
\begin{Example}
If $M$ has $\varphi$-cusps with $\varphi(x) = -\mu\log(x)$ as in Example~\ref{ex mu log x}, then $M$ has weakly admissible cusps if and only if $\mu \geq 0$. This includes the case $\mu=0$ of a cylindrical end, relevant to the Atiyah--Patodi--Singer index theorem. For cusp metrics of this form, we have the implications
\[
\text{finite volume} \Rightarrow \text{strongly admissible cusps} \Rightarrow \text{weakly admissible cusps}.
\]
For general metrics of the form~\eqref{eq g phi}, only the second implication always holds.
\end{Example}

\begin{Example}
If $\varphi$ is periodic, then $\varphi$-cusps are never strongly admissible. But if $\varphi(x) = \tilde \varphi(\sin(x))$ for some $\tilde \varphi \in C^{\infty}([-1,1])$, then $\varphi$-cusps are weakly admissible if $|\tilde \varphi'| < b$.
\end{Example}

\subsection{Spectral theory for Sturm--Liouville operators}

Let $q$ be a real-valued, continuous function on the closed half-line $[0, \infty)$.
A crucial role will be played by the spectral theory of Sturm--Liouville operators of the form
\[
\Delta_{q}:= -\frac{{\rm d}^2}{{\rm d}y^2} + q,
\]
on $[0, \infty)$.
We briefly review this theory here, and refer to \cite{Levitan91, Titchmarsh62} for details.

For $\nu \in \C$, let $\theta_{\nu} \in C^{\infty}([0,\infty))$ be the unique solution of
\[
\Delta_{q}\theta_{\nu} = \nu \theta_{\nu},
\]
such that
\begin{equation} \label{eq Dirichlet}
\theta_{\nu}(0) = 0, \qquad
\theta_{\nu}'(0) = 1.
\end{equation}
The theory extends to more general boundary conditions, but we will only use the Dirichlet case. For $f \in C^{\infty}_c(0,\infty)$ and $\nu \in \R$, define the generalised Fourier transform
\[
\cF_q(f)(\nu):= \int_0^{\infty} f(y) \theta_{\nu} (y)\, {\rm d}y.
\]
 For a function $\rho\colon \R \to \C$, let $L^2(\R, {\rm d}\rho)$ be the space of square-integrable functions with respect to the measure ${\rm d}\rho$, in the sense of Stieltjes integrals. Note that this measure may have singular points, if $\rho$ is discontinuous.
 \begin{Theorem}\label{thm spec decomp}
 There exists a unique increasing function $\rho\colon \R \to \R$ with the following prop\-erties:
 \begin{itemize}\itemsep=0pt
 \item[$(a)$]
 The map $f \mapsto \cF_q(f)$ extends to a unitary isomorphism from $L^2([0,\infty))$ onto $L^2(\R, {\rm d}\rho)$.
 \item[$(b)$]
 For all continuous $f \in L^2([0,\infty))$ such that the integrands on the right-hand side are well-defined and the integral converges uniformly in $y$ in compact intervals,
 \begin{equation} \label{eq inversion}
 f(y) = \int_{\R} \cF_q(f)(\nu) \theta_{\nu}(y)\,{\rm d}\rho(\nu).
 \end{equation}
 \end{itemize}
 \end{Theorem}
 See, for example, \cite[Theorems~2.1.1 and 2.1.2]{Levitan91}.

The spectral measure ${\rm d}\rho$ can be computed as follows.
\begin{Proposition}\label{prop rho Titch}
For $\nu \in \C$,
let $\theta_1( \relbar, \nu)$ and $\theta_2( \relbar, \nu)$ be the solutions of $\Delta_{q}\theta_j = \nu \theta_j$ on $[0, \infty)$, such that
\[
\theta_1(0,\nu) = 0,\qquad
\theta_1'(0,\nu) = 1,\qquad
\theta_2(0,\nu) = -1,\qquad
\theta_2'(0,\nu)= 0,
\]
where the prime denotes the derivative with respect to the first variable. There is a function~$f$ on the upper half-plane in $\C$, with negative imaginary part, such that for all $\nu$ in the upper half-plane, $\theta_2(\relbar, \nu)+ f(\nu) \theta_1(\relbar, \nu) \in L^2([0,\infty))$. And for all $\nu \in \R$,
\begin{equation} \label{eq expr rho}
\rho(\nu) = \frac{1}{\pi} \lim_{\delta \downarrow 0} \int_0^{\nu} -\Imag(f(\nu'+{\rm i}\delta))\, {\rm d}\nu'.
\end{equation}
\end{Proposition}
See \cite[Theorem~2.4.1]{Levitan91}, or~\cite[Chapter~3]{Titchmarsh62}, in particular
 Lemma~3.3 and the theorem on p.~60. The factor $\frac{1}{\pi}$ in~\eqref{eq expr rho} corresponds to the same factor in the inversion formula in the theorem on p.~60 of \cite{Titchmarsh62}, which is not present in~\eqref{eq inversion}.

\subsection[An index theorem for manifolds with admissible varphi-cusps]{An index theorem for manifolds with admissible $\boldsymbol{\varphi}$-cusps}

We suppose, as before, that $M$ is complete and has weakly admissible $\varphi$-cusps. Let $a$ be as in Definition~\ref{def cusps}. Consider the function $\xi\colon (a,\infty) \to (0, \infty)$ defined by
\begin{equation} \label{eq def xi}
\xi(x) := \int_{a}^x {\rm e}^{\varphi(x')}\, {\rm d}x'.
\end{equation}
Then $\xi$ is injective because its derivative is positive, and surjective by Lemma~\ref{lem phi complete}. For $\lambda \in \R$, define the functions $q_{\lambda}^{\pm} \in C^{\infty}(0,\infty)$ by
\begin{equation} \label{eq def q lambda}
q_{\lambda}^{\pm}(y) := \lambda \bigl((\lambda \pm \varphi') {\rm e}^{-2\varphi} \bigr)\circ \xi^{-1}.
\end{equation}
(We only use $q_{\lambda}^+$
in the current section, but $q_{\lambda}^-$
will also be used in Section~\ref{sec cusp contr}.)
In the case where $q = q^{\pm}_{\lambda}$, we write
\begin{equation} \label{eq Delta lambda}
 \Delta^{\pm}_{\lambda} := \Delta_{q^{\pm}_{\lambda}}.
\end{equation}
We then write $\rho^{\lambda, \pm}$ for the function $\rho$ in Theorem~\ref{thm spec decomp}, and $\theta_{\nu}^{\lambda, \pm}$ for the function $\theta_{\nu}$ in Theorem~\ref{thm spec decomp}.

\begin{Example}
Suppose that $\varphi(x) = -\mu\log(x)$, for $\mu \in \R$. If $\mu = 1$, then
\[
q_{\lambda}^{\pm}(y) = a^2 \lambda^2 {\rm e}^{2y} \mp a\lambda {\rm e}^y.
\]
If $\mu \not = 1$, then
\[
q_{\lambda}^{\pm}(y) = \lambda^2 \left( (1-\mu)y+a^{1-\mu} \right)^{\frac{2\mu}{1-\mu}} \mp \lambda \mu \left( (1-\mu)y + a^{1-\mu}\right)^{\frac{2\mu - 1}{1-\mu}}.
\]
\end{Example}

For a finite-dimensional vector space $V$ with a given representation by $G$, we write $\tr(g|_V)$ for the trace of the action by $g$ on $V$.
\begin{Definition}\label{def cusp cpt}
The \emph{$g$-delocalised $\varphi$-cusp contribution} associated to $D_N^+$ and a number ${a'>a}$~is
\begin{align}
\eta^{\varphi}_g\big(D_N^+, a'\big)&=2 {\rm e}^{-p\varphi(a')} \int_{0}^{\infty}
\sum_{\lambda \in \spec(D_N^+)}\sgn(\lambda) \tr\big(g|_{\ker(D_N^+ - \lambda)}\big)
\label{eq cusp cpt}
\\
&\phantom{=} \times
\int_{\R} {\rm e}^{-s\nu}
\theta^{|\lambda|, +}_{\nu} (\xi(a'))
 \big(\big(\theta^{|\lambda|, +}_{\nu} \big)'(\xi(a')) +|\lambda| {\rm e}^{-\varphi(a')} \theta^{|\lambda|, +}_{\nu} (\xi(a'))
 \big)\,{\rm d}\rho^{|\lambda|, +}(\nu)\nonumber
 \, {\rm d}s,
\end{align}
if the right-hand side converges.
\end{Definition}
We will see in Theorem~\ref{thm index cusps} that~\eqref{eq cusp cpt} indeed converges in the situations we consider.

The notation $\eta^{\varphi}_g$ and the factor $2$ in the definition are motivated by the fact that $\eta^0_e$ is the usual $\eta$-invariant of \cite{APS1}, and, more generally, $\eta^0_g$ is the delocalised $\eta$-invariant \cite{Donnelly}. See Proposition~\ref{prop cusp contr cylinder}.

\begin{Definition}
The spectrum of $D_N^+$ is \emph{$g$-symmetric} if for all $\lambda \in \R$,
\[
\tr\big(g|_{\ker(D_N^+ - \lambda)}\big) = \tr\big(g|_{\ker(D_N^+ + \lambda)}\big).
\]
\end{Definition}
Note that if $g=e$, then the spectrum of $D_N$ is {$e$-symmetric} precisely if it is symmetric around zero, including multiplicities.

For a $G$-equivariant, odd-graded, self-adjoint Fredholm operator $F$ on $L^2(S)$, we denote its restrictions to even- and odd-graded sections by $F^+$ and $F^-$, respectively. Then
\[
\ind_G(F) := [\ker(F^+)] - [\ker(F^-)] \in R(G)
\]
is the classical equivariant index of $F^+$, in the representation ring $R(G)$ of $G$. We denote the value of its character at $g$ by
\[
\ind_G(F)(g) = \tr\big(g|_{\ker(F^+)}\big) - \tr\big(g|_{\ker(F^-)}\big).
\]
Let $\AS_g(D)$ be the Atiyah--Segal--Singer integrand associated to $D$, see, for example, \cite[Theo\-rem~6.16]{BGV} or~\cite[Theo\-rem~3.9]{ASIII}. It is a differential form of mixed degree on the fixed-point set~$M^g$ of~$g$. The connected components of $M^g$ may have different dimensions, and the integral of $\AS_g(D)$ over $M^g$ is defined as the sum over these connected components of the integral of the component of the relevant degree.
\begin{Theorem}[index theorem on manifolds with admissible cusps]\label{thm index cusps}
Suppose that $M$ has weakly admissible $\varphi$-cusps, that $D|_C = D_{\varphi}$ as in Proposition~$\ref{prop Dphi}$, and that $D_N$ is invertible.
Then $D$ is Fredholm, the cusp contribution~\eqref{eq cusp cpt} converges for all $a'>a$, and
\begin{equation} \label{eq index cusps}
\ind_G(D)(g) = \int_{Z^g \cup (N^g \times (a,a'])} \AS_g(D) - \frac{1}{2}\eta_g^{\varphi}\big(D_N^+,a'\big).
\end{equation}
Furthermore, $\lim_{a' \downarrow a}\eta_g^{\varphi}\big(D_N^+,a'\big)$ converges, and
\begin{equation} \label{eq index cusps 2}
\ind_G(D)(g) = \int_{Z^g} \AS_g(D) - \frac{1}{2}\lim_{a' \downarrow a} \eta_g^{\varphi}\big(D_N^+,a'\big).
\end{equation}
If $\int_{M^g} \AS_g(D)$ converges, then $\lim_{a' \to \infty}\eta_g^{\varphi}\big(D_N^+,a'\big)$ converges, and
\begin{equation} \label{eq index cusps 3}
\ind_G(D)(g) = \int_{M^g} \AS_g(D) - \frac{1}{2}\lim_{a' \to \infty} \eta_g^{\varphi}\big(D_N^+,a'\big).
\end{equation}

If the spectrum of $D_N^+$ is $g$-symmetric, then $\eta_g^{\varphi}\big(D_N^+, a'\big) = 0$ for all $a'>a$.
If $M$ has strongly admissible $\varphi$-cusps, then $\Delta_{|\lambda|}^+$ has discrete spectrum for all $\lambda$, so the integral over $\nu$ in~\eqref{eq cusp cpt} becomes a sum.
\end{Theorem}

The case \eqref{eq index cusps 3} of Theorem~\ref{thm index cusps} applies in some relevant special cases. The following fact follows from Proposition~\ref{prop cusp contr cylinder}. This also follows from the fact that the $\hat A$-form on the cylinder $N \times (a, \infty)$ is zero.

\begin{Lemma}
In the setting of Theorem~$\ref{thm index cusps}$,
if $\varphi = 0$, then~\eqref{eq index cusps 3} applies.
\end{Lemma}

The case \eqref{eq index cusps 3} also applies if $M^g$ is compact; this is equivalent to $g$ having no fixed points on $N$.
\begin{Lemma}In the setting of Theorem~$\ref{thm index cusps}$, if $M^g$ is compact, then~\eqref{eq index cusps 3} applies.
\end{Lemma}
If $M$ has strongly admissible cusps, so that $\Delta_{|\lambda|}^+$ has discrete spectrum, and we assume~for~sim\-plicity that the eigenspaces are one-dimensional, then unitarity of \smash{$\cF_{q_{|\lambda|}^+}$} implies that the measure~with respect to ${\rm d}\rho_{\nu}^{|\lambda|, +}$ of every point $\nu \in \spec(\Delta_{|\lambda|}^+)$ is $1/\|\theta_{\nu}^{|\lambda|, +}\|_{L^2}^{2}$. Then~\eqref{eq cusp cpt} becomes
\begin{align*}
\eta^{\varphi}_g\big(D_N^+, a'\big) &=
2 {\rm e}^{-p\varphi(a')}
 \int_{0}^{\infty}
\sum_{\lambda \in \spec(D_N^+)}\sgn(\lambda) \tr\big(g|_{\ker(D_N^+ - \lambda)}\big)
\\
&\phantom{=}\times\!\!\!\!\sum_{\nu \in \spec(\Delta_{|\lambda|}^+)}\!
\frac{{\rm e}^{-s\nu}}{\|\theta_{\nu}^{|\lambda|, +}\|_{L^2}^2}
\theta^{|\lambda|, +}_{\nu} (\xi(a'))
 \big( \big(\theta^{|\lambda|, +}_{\nu} \big)'(\xi(a')) +|\lambda| {\rm e}^{-\varphi(a')} \theta^{|\lambda|, +}_{\nu} (\xi(a')) \big) \, {\rm d}s.
\end{align*}
It is generally not possible to normalise the eigenfunctions $\theta_{\nu}^{|\lambda|, +}$ so that their $L^2$-norms are $1$, because they should satisfy~\eqref{eq Dirichlet}.

\begin{Remark}
The expression~\eqref{eq cusp cpt} can be extended to $a' = a$, but then it equals zero because the functions $\theta_{\nu}^{|\lambda|, +}$ satisfy~\eqref{eq Dirichlet}. So the limit $\lim_{a' \downarrow a} \eta_g^{\varphi}\big(D_N^+,a'\big)$ on the right-hand side of~\eqref{eq index cusps 2} is generally different from the version of~\eqref{eq cusp cpt} with $a' = a$. See Remark~\ref{rem cyl zero at a} for an example, and also~\cite[Remarks~2.4 and~5.12]{HW21a}.
\end{Remark}

\begin{Remark}As is the case for the classical $\eta$-invariant, the cusp contribution~\eqref{eq cusp cpt} measures (an equivariant version of) \emph{spectral asymmetry}, in the sense that it is zero if the spectrum of $D_N^+$ is $g$-symmetric.
 It is intriguing that even in this more general setting, symmetry or asymmetry of the spectrum of $D_N$ determines if a contribution ``from infinity'' is required in index theorems on manifolds of the type we consider. Spectral asymmetry of $D_N^+$ can be measured in different ways, and apparently, the cusp shape function $\varphi$ determines what measure of spectral asymmetry is relevant for an index theorem on $M$.
\end{Remark}

\subsection{Relations to other results}\label{sec other results}

In some cases, the cusp metric~\eqref{eq g phi} can be transformed to a metric of the form $\rho^2 B_N + {\rm d}u^2$, for a function $\rho$ of a radial coordinate $u$. This helps to clarify the relations between Theorem~\ref{thm index cusps} and the results in \cite{Baier02, BB03, BBC12, Stern84}.

Suppose that $\varphi'$ has no zeroes.
Let $\rho$ be a positive, smooth function defined on $(\tilde a, \tilde b)$, with $\tilde a:= \rho^{-1}\big({\rm e}^{\varphi(a)}\big)$ and $\tilde b:= \lim_{x \to \infty} \rho^{-1}\big({\rm e}^{\varphi(x)}\big)$. Suppose that, on this interval,
\begin{equation} \label{eq rho phi}
\rho' = \varphi'\circ \varphi^{-1}\circ \log ( \rho).
\end{equation}
Then the map $(n, u) \mapsto \bigl(n, \varphi^{-1}(\log(\rho(u))) \bigr)$ is an isometry from $N \times \big(\tilde a, \tilde b\big)$, with Riemannian metric $\rho^2 B_N + {\rm d}u^2$, onto $N \times (a, \infty)$, with Riemannian metric ${\rm e}^{2\varphi} \big(B_N + {\rm d}x^2\big)$.

In the case where $\varphi(x) =-\mu\log(x)$, for $\mu \in [0,1]$, the function
\[
\rho(u) = (1-\mu)^{\mu/(\mu-1)} u^{\mu/(\mu-1)}
\]
is a solution of~\eqref{eq rho phi} if $\mu \not= 1$. If $\mu = 1$, then a solution is $\rho(u) = {\rm e}^{-u}$, a well-known alternative form of hyperbolic metrics.

In his thesis \cite{Stern84}, Stern computed the index of the signature operator on finite-volume manifolds with cusps $N \times (\tilde a, \infty)$ with metric $\rho^2 B_N + {\rm d}u^2$, under certain growth conditions on $\rho$ and its derivatives. These conditions hold in the case where $\varphi(x) =-\mu\log(x)$, for $\mu \in [0,1]$.

Baier \cite{Baier02} developed
index theory on $\Spin$-manifolds with cusps $N\times (\tilde a, \infty)$, with metrics of the form $\rho^2 B_N + {\rm d}u^2$.
He considered indices of $\Spin$-Dirac operators, and proved
\begin{enumerate}\itemsep=0pt
\item[(1)] vanishing of the index if there is a $c>1$ such that $u^{-c} \rho(u)$ has a positive lower bound for large $u$,
\item[(2)] an index formula if there is a $c<1$ such that $u^{-c} \rho(u)$ has an upper bound for large $u$,
\item[(3)] inequalities satisfied by the index in other cases.
\end{enumerate}

 In cylinder case $\mu = 0$ and the hyperbolic case $\mu = 1$, the second of Baier's results applies. For $\mu \in (0,1)$, neither condition in the first two cases holds, so the inequalities for the index in the third case apply.

In cases where both Theorem~\ref{thm index cusps} and Baier's index formula apply, Baier's formula has a~simpler contribution from outside $Z$ (the classical $\eta$-invariant), but a more complicated contribution from inside $Z$, involving the dimension of the kernel of $D^-|_Z$ with Dirichlet boundary conditions at $N$.

Ballmann and
Br\"uning \cite{BB03} considered two-dimensional manifolds of finite area, with cusps of a type that includes
 cusps $S^1 \times (\tilde a, \infty)$ with metrics of the form $\rho^2 B_N + {\rm d}u^2$ mentioned above. Then they obtained an explicit index formula for Dirac operators.

 Ballmann,
Br\"uning and Carron \cite{BBC12} studied manifolds with cusps of a different but related form. In their setting, the manifold's ends are diffeomorphic to $N \times (\tilde a, \infty)$ via the gradient flow of a function with certain properties. They obtain a general index theorem in this setting, and more concrete index theorems in the case where the metric on the ends is cuspidal. A metric of the form $\rho^2 B_N + {\rm d}u^2$, with $\rho$ related to $\varphi$ via~\eqref{eq rho phi}, is cuspidal if $\varphi(x) =-\log(x)$, but not if $\varphi(x) =-\mu \log(x)$ for $\mu \in [0,1)$.

Vaillant \cite{Vaillant01} obtained an index theorem for manifolds with fibred versions of hyperbolic cusps. In the case of hyperbolic cusps, he showed that the contribution from infinity equals zero. It is an interesting question to find conditions on $\varphi$ that imply vanishing of the cusp contribution in general.

For some results on the spectrum of Dirac or Laplace operators on manifolds with hyperbolic cusps, see \cite{Baer00} for $\Spin$-Dirac operators and \cite{MP90} for the Hodge Laplacian. See \cite{Moroianu07} for the Hodge Laplacian on manifolds with a generalisation of hyperbolic cusps.

See \cite{Moscovici82} for finite-dimensionality of the kernels of Dirac operators on finite-volume hyperbolic locally symmetric spaces. The Fredholm property in Theorem~\ref{thm index cusps} is a stronger version of this, under the condition that $D_N$ is invertible. The latter condition can be removed as in Section~\ref{sec DN not inv}.

In APS-type index theorems where the metric does not have a product structure near the boundary, the contribution from the boundary can usually be written as a local contribution involving a transgression form, and a spectral contribution, the usual $\eta$-invariant. See, e.g., \cite{Gilkey93, Grubb92, Salomonsen98}, or \cite{BM19} for an equivariant version. It is an interesting question if the cusp contribution in Theorem~\ref{thm index cusps}, defined in terms of the spectrum of the operator $D_N^+$, can be decomposed in a similar way, into a local contribution from the manifold $N$ and a possibly simpler spectral contribution.

\section{Proof of Theorem~\ref{thm index cusps}}

We first state an index theorem from \cite{HW21a} for Dirac operators that are invertible outside a~compact set, see Theorem~\ref{thm index general}. Then we deduce Theorem~\ref{thm index cusps} from this result.

\subsection{Dirac operators that are invertible at infinity}\label{sec dirac ops}

We review the geometric setting and notation needed to formulate Theorem 2.2 in \cite{HW21a}, this is Theorem~\ref{thm index general} below. The text leading up to Theorem~\ref{thm index general} is a slight reformulation of corresponding material from \cite{HW21a}. Compared to Theorem~\ref{thm index cusps}, some assumptions in
 the main result from \cite{HW21a}, Theorem~\ref{thm index general}, are weaker, but there are additional assumptions on $D$, such as invertibility at infinity.
 Part of the proof of Theorem~\ref{thm index cusps} is to show that these additional assumptions hold in the setting of manifolds with admissible cusps.

 We assume in this subsection and the next that $M$ has $\varphi$-cusps, but not that these cusps are weakly or strongly admissible.

We say that $D$ is \emph{invertible at infinity} if there are a $G$-invariant compact subset $Z\subset M$ and a constant $b>0$ such that for all $s \in \Gamma^{\infty}_c(S)$ supported in $M \setminus Z$,
\begin{equation} \label{eq D inv infty}
\|Ds\|_{L^2} \geq b \|s\|_{L^2}.
\end{equation}
We assume in this subsection and the next that $D$ is {invertible at infinity}, and that the set $Z$ may be taken as in Definition~\ref{def cusps}. (In \cite{HW21a}, we took $a=0$, now we allow general $a \geq 0$ for consistency with Theorem~\ref{thm index cusps}.)

For $k=0, 1, 2, \dots$, let $W^k_D(S)$ be the completion of $\Gamma^{\infty}_c(S)$ in the inner product
\begin{equation*}
(s_1, s_2)_{W^k_D} := \sum_{j=0}^k \big(D^js_1, D^js_2\big)_{L^2}.
\end{equation*}
Because $D$ is invertible at infinity, it is Fredholm as an operator
\[
D\colon\ W^1_D(S) \to L^2(S).
\]
See \cite[Theorem 2.1]{Anghel93b} or \cite[Theorem 3.2]{Gromov83}.

Rather than the more specific form~\eqref{eq DC Dphi} of $D$ on $C = M \setminus Z$, we assume that
\begin{equation} \label{eq D on U}
D|_C = \sigma \bigg(f_1 \frac{\partial}{\partial x} + f_2 D_N + f_3 \bigg),
\end{equation}
where
\begin{itemize}\itemsep=0pt
\item $\sigma \in \End(S|_N)^G$ interchanges $S^+|_N$ and $S^-|_N$,
\item $f_1, f_2, f_3 \in C^{\infty}(a,a+2)$,
\item $D_N$ is a $G$-equivariant, invertible Dirac operator on $S|_N$ that preserves the grading.
\end{itemize}

Consider the vector bundle
\[
S_C := S|_{C} \to C.
\]
For $k \in \N$ at least $1$, consider the Sobolev space
\[
W^k_D(S_C) := \big\{s|_C; s \in W^k_D(S)\big\}.
\]
(See \cite{BB12} for other constructions of such Sobolev spaces on manifolds with boundary.)
We denote the subspaces of even- and odd-graded sections by $W^k_D\big(S_C^{\pm}\big)$, respectively.

If $s \in W^{k+1}_D(S)$, then the restriction of $Ds \in W^k_D(S)$ to the interior of $C$ is determined by the restriction of $s$ to the interior of $C$. Since $k \geq 1$, the restriction of $Ds$ to the interior of $C$ has a unique extension to $C$. So $Ds|_C$ is determined by $s|_C$. In this way, $D$ gives a well-defined, bounded operator from $W^{k+1}_D(S_C)$ to $W^k_D(S_C)$, which we denote by
 $D_C$.

Because $N$ is compact, $D_N$ has discrete spectrum.
Let $D_N^{\pm}$ be the restriction of $D_N$ to sections of $S^{\pm}|_N$.
 Let $L^2\big(S^+|_N\big)_{>0}$ be the direct sum of these eigenspaces for positive eigenvalues (recall that $0$ is not an eigenvalue).
 Consider the orthogonal projection
\begin{equation} \label{eq def P+}
P^+\colon\ L^2\big(S^+|_N\big) \to L^2\big(S^+|_N\big)_{>0}.
\end{equation}
(In \cite{HW21a}, more general spectral projections are allowed, but this is the one relevant to the current setting.)
We will also use the projection
\[
P^- := \sigma_+ P^+ \sigma_+^{-1} \colon\ L^2\big(S^-|_N\big) \to \sigma_+L^2\big(S^+|_N\big)_{>0},
\]
where $\sigma_+:= \sigma|_{(S^+|_N)}$.
In general, $P^-$ is not necessarily a spectral projection for $D_N^-$, but in the setting that is relevant to us, we have $\sigma D_N^+ = -D_N^- \sigma$, so $P^-$ is projection onto the \emph{negative} eigenspaces of $D_N^-$; see~\eqref{eq DN sigma}.
We combine $P^+$ and $P^-$ to an orthogonal projection
\begin{equation} \label{eq def P}
P := P^+ \oplus P^-\colon \ L^2(S|_N) \to L^2\big(S^+|_N\big)_{>0} \oplus
\sigma_+L^2\big(S^+|_N\big)_{>0}.
\end{equation}
We
 will sometimes omit the superscripts $\pm$ from $P^{\pm}$.

Consider the spaces
\begin{equation} \label{eq def WnD}
\begin{split}
&W^1_D\big(S_C^+; P\big):= \big\{s \in W^1_D\big(S_C^+\big); P^+(s|_N) = 0\big\},\\[1mm]
&W^1_D\big(S_C^-;1- P\big) := \big\{s \in W^1_D\big(S_C^-\big); (1-P^-)(s|_N) = 0\big\},\\[1mm]
&W^2_D\big(S_C^+; P\big) := \big\{s \in W^2_D\big(S_C^+\big); P^+(s|_N) = 0, (1-P^-)\big(D^+_Cs|_N\big) = 0\big\},\\[1mm]
&W^2_D\big(S_C^-; 1-P\big) := \big\{s \in W^2_D\big(S_C^-\big); (1-P^-)(s|_N) = 0, P^+\big(D^-_Cs|_N\big) = 0\big\}.
\end{split}
\end{equation}
Here we use the fact that there are well-defined, continuous restriction/extension maps
\[
{W^1_D(S_C) \to L^2(S|_N)}.
\]
We assume that
\begin{itemize}\itemsep=0pt
\item the operators
\begin{equation} \label{eq DC+}
D^+_C\colon\ W^1_D\big(S_C^+; P\big) \to L^2\big(S_C^-\big)
\end{equation}
and
\begin{equation} \label{eq DC-}
D^-_C\colon\ W^1_D\big(S_C^-; 1-P\big) \to L^2\big(S_C^+\big)
\end{equation}
are invertible; and
\item the operator $D^-_C D_C^+$ on $L^2\big(S_C^+\big)$, with domain $W^2_D\big(S_C^+; P\big)$, and the operator $D^+_C D_C^-$ on $L^2\big(S_C^-\big)$, with domain $W^2_D\big(S_C^-; 1-P\big)$, are self-adjoint.
\end{itemize}

\subsection{An index theorem} \label{sec index thm general}

Under the assumptions in Section~\ref{sec dirac ops}, we state an index theorem using the following ingredients.

For $t>0$, let ${\rm e}_{P}^{-t D_C^- D_C^+}$ be the heat operator for the operator
 $D^-_C D_C^+$ on $L^2\big(S_C^+\big)$, with domain $W^2_D\big(S_C^+; P\big)$.
By~\cite[Lemma 4.7]{HW21a}, the operator ${\rm e}_{P}^{-t D_C^- D_C^+}D_C^-$ has a smooth kernel $ \lambda_t^{P}$.
 The \emph{contribution from infinity} associated to $D_C$ and $a' \in (a,a+2)$ is
\begin{equation} \label{eq def At}
A_{g}(D_C, a') := -f_1(a') \int_0^{\infty}
\int_N \tr(g \lambda_s^{P}(g^{-1}n, a'; n, a'))\, {\rm d}n\,{\rm d}s,
\end{equation}
defined whenever
 the integral in~\eqref{eq def At} converges.

The following result is a combination of Theorem 2.2 and Corollary 2.3 in \cite{HW21a}.
\begin{Theorem}[index theorem for Dirac operators invertible at infinity]\label{thm index general}
 For all $a' > a$, the quantity~\eqref{eq def At} converges, and
\begin{equation} \label{eq index general}
\ind_G(D)(g) = \int_{Z^g \cup (N^g \times (a,a'])} \AS_g(D)+ A_{g}(D_C, a').
\end{equation}
Furthermore, the limit $\lim_{a' \downarrow a}A_{g}(D_C, a')$ converges, and
\[
\ind_G(D)(g) = \int_{Z^g} \AS_g(D)+ \lim_{a' \downarrow a}A_{g}(D_C, a').
\]
\end{Theorem}

Theorem~\ref{thm index cusps} follows from Theorem~\ref{thm index general} because of the following three propositions.
\begin{Proposition} \label{prop ind cusps cond}
In the setting of Theorem~$\ref{thm index cusps}$, the conditions of Theorem~$\ref{thm index general}$ hold.
\end{Proposition}

\begin{Proposition} \label{prop ind cusp contr}
In the setting of Theorem~$\ref{thm index cusps}$, the cusp contribution~$\eqref{eq cusp cpt}$ converges for all $a'>a$, and
\begin{equation*} 
A_{g}(D_C, a') = -\frac{1}{2}\eta_g^{\varphi}\big(D_N^+, a'\big).
\end{equation*}
If the spectrum of $D_N$ is $g$-symmetric, then the right-hand side is zero.
\end{Proposition}

\begin{Proposition}\label{prop discrete spec}
If $M$ has strongly admissible $\varphi$-cusps, then the operators $\Delta_{|\lambda|}^{+}$ have discrete spectrum for all $\lambda \in \spec\big(D_N^+\big)$.
\end{Proposition}
\begin{proof}
It follows from the definition~\eqref{eq def q lambda} of $q_{\lambda}^{+}$ that for all $\lambda \not=0$,
\[
\lim_{y \to \infty}q_{|\lambda|}^{+}(y) = \infty
\]
if $\varphi$ satisfies~\eqref{eq strong adm}. It follows that $\Delta_{|\lambda|}^{+}$ has discrete spectrum, see~\cite[Theorem~1.3.1, Lemma~3.1.1 and equation~(1.3)]{Levitan91}.
\end{proof}

In the rest of this section, we prove Propositions~\ref{prop ind cusps cond} (at the end of Section~\ref{sec adjoints}) and~\ref{prop ind cusp contr} (at the end of Section~\ref{sec cusp contr}), and thus Theorem~\ref{thm index cusps}. (The case~\eqref{eq index cusps 3} follows immediately from the case~\eqref{eq index cusps}.)

\subsection{Transforming Dirac operators}

We return to the setting of Section~\ref{sec results}, where $D$ is of the form in Proposition~\ref{prop Dphi}.

To compute~\eqref{eq def At}, we use a Liouville-type transformation to relate $D_C$ and $D_C^2$ to simpler operators. At the same time, this allows us to transform the Riemannian density of $B_{\varphi}$ to a~product density.

Define $\Phi \in C^{\infty}(a, \infty)$ by
\begin{equation} \label{eq def Phi}
\Phi(x) := {\rm e}^{-\frac{p-1}{2}\varphi(x)}.
\end{equation}
(Recall that $p = \dim(M)$.)
Consider the vector bundle $\widetilde{S}_C := S|_N \times (0,\infty) \to N \times (0, \infty)$. For a section $s$ of $\widetilde{S}_C$ and $n \in N$ and $x \in (a,\infty)$, define
\[
(Ts)(n,x) := \Phi(x) s(n, \xi(x))
\]
with $\xi$ as in~\eqref{eq def xi}.
\begin{Lemma}\label{lem T isom}
The operator $T$ defines a $G$-equivariant unitary isomorphism
\[
T\colon \ L^2\big(\widetilde{S}_C, \vol_{B_N} \otimes {\rm d}y\big) \to L^2\big(S|_C, \vol_{B_{\varphi}}\big).
\]
\end{Lemma}
\begin{proof}
We have
$\vol_{B_{\varphi}} = \vol_{B_N} \otimes {\rm e}^{p\varphi} {\rm d}x$.
And by a substitution $y = \xi(x)$, and the equality $\xi' = {\rm e}^{\varphi}$, we have for all $f \in L^2( (0,\infty), {\rm d}y)$,
\begin{equation*}
\|\Phi \cdot (f \circ \xi)\|^2_{L^2( (a,\infty), {\rm e}^{p\varphi} {\rm d}x)} =
 \int_{a}^{\infty} |f(\xi(x))|^2 {\rm e}^{\varphi(x)}\, {\rm d}x = \|f\|^2_{L^2( (0,\infty), {\rm d}y)}.\tag*{\qed}
\end{equation*}
\renewcommand{\qed}{}
\end{proof}

Consider the function $h:= {\rm e}^{-\varphi \circ \xi^{-1}} \in C^{\infty}(0, \infty)$, and the Dirac operator
\begin{equation*}
\widetilde{D}_C := \sigma\bigg( \frac{\partial}{\partial y} + h D_N\bigg)
\end{equation*}
on $\Gamma^{\infty}\big(\widetilde{S}_C\big)$. Here $y$ is the coordinate in $(0, \infty)$. Let $D_C$ be as in Section~\ref{sec dirac ops}, but viewed as acting on smooth sections. It is given by~\eqref{eq Dphi}.
\begin{Lemma}\label{lem T Dpsi}
The following diagram commutes:
\[
\xymatrix{
\Gamma^{\infty}(S_C) \ar[r]^-{D_{C}} & \Gamma^{\infty}(S_C) \\
\Gamma^{\infty}\big(\widetilde{S}_C\big) \ar[u]^-{T} \ar[r]^-{\widetilde{D}_C} & \Gamma^{\infty}\big(\widetilde{S}_C\big) \ar[u]^-{T}.
}
\]
\end{Lemma}
\begin{proof}
This is a direct computation, based on~\eqref{eq Dphi} and~\eqref{eq DC Dphi}.
\end{proof}

It will be convenient to identify $S^-|_N \cong S^+|_N$ via $\sigma$ in~\eqref{eq DC Dphi}. Because signs and gradings are important in what follows, it is worth being explicit about details here.
We write
 $\tau:= \sigma|_{S^+|_N} \times 1\colon \widetilde{S}_C^+ \to \widetilde{S}_C^-$, and consider the isomorphism
\begin{equation} \label{eq S0++}
1 \oplus \tau\colon\
\widetilde{S}_C^+ \oplus \widetilde{S}_C^+ \xrightarrow{\cong}\widetilde{S}_C.
\end{equation}
The operator $D_N$ preserves the grading on $S|_N$; let $D_N^{\pm}$ be its restrictions to even and odd-graded sections, respectively. Because $D_N$ anticommutes with $\sigma$, we have
\begin{equation} \label{eq DN sigma}
D_N^- \circ \tau = -\tau \circ D_N^+.
\end{equation}
We will use the operators
\[
\widetilde{D}_C^{\pm} := \pm \frac{\partial}{\partial y} +hD_N^+\colon\ \Gamma^{\infty}\big(\widetilde{S}_C^+\big) \to \Gamma^{\infty}\big(\widetilde{S}_C^+\big).
\]

\begin{Lemma}\label{lem D0pm}
Under the isomorphism~\eqref{eq S0++}, the operator $\widetilde{D}_C$ corresponds to the operator
\[
\begin{pmatrix}
0 & \widetilde{D}_C^-\\
\widetilde{D}_C^+ & 0
\end{pmatrix}
\qquad \text{on} \quad \Gamma^{\infty}\big(\widetilde{S}_C^+ \oplus \widetilde{S}_C^+\big).
\]
\end{Lemma}

\begin{proof}
Consider the vector bundle endomorphism
\[
\Sigma := \begin{pmatrix}0&-1\\1&\hphantom{-}0\end{pmatrix}\!\colon\ \widetilde{S}_C^+ \oplus \widetilde{S}_C^+\to \widetilde{S}_C^+ \oplus \widetilde{S}_C^+.
\]
Then the diagram{\samepage
\[
\xymatrix{
\widetilde{S}_C \ar[r]^-{\sigma} & \widetilde{S}_C \\
\widetilde{S}_C^+ \oplus \widetilde{S}_C^+ \ar[r]^-{\Sigma} \ar[u]^-{1\oplus \tau}& \widetilde{S}_C^+ \oplus \widetilde{S}_C^+\ar[u]^-{1\oplus \tau}
}
\]
commutes.}

Because of~\eqref{eq DN sigma} and the fact that
the operator $\frac{\partial}{\partial y}$ commutes with $\tau$,
\begin{align*}
\bigg(\frac{\partial}{\partial y}+hD_N\bigg) \circ (1\oplus \tau) &=
\begin{pmatrix}
\frac{\partial}{\partial y}+hD_N^+&0\\
0&\frac{\partial}{\partial y}+hD_N^-
 \end{pmatrix}
\begin{pmatrix}
1 & 0 \\
0 & \tau
\end{pmatrix}\\
&=
\begin{pmatrix}
1 & 0 \\
0 & \tau
\end{pmatrix}\begin{pmatrix}
\frac{\partial}{\partial y}+hD_N^+&0\\
0&\frac{\partial}{\partial y}-hD_N^+
 \end{pmatrix}\!.
\end{align*}
 We find that under the isomorphism~\eqref{eq S0++}, the operator $\widetilde{D}_C$ corresponds to
\begin{equation*}
\Sigma
\begin{pmatrix}
\frac{\partial}{\partial y}+hD_N^+&0\\
0&\frac{\partial}{\partial y}-hD_N^+
 \end{pmatrix}
 =\begin{pmatrix}
0 & -\frac{\partial}{\partial y}+hD_N^+\\
\frac{\partial}{\partial y}+hD_N^+& 0
 \end{pmatrix}\!.\tag*{\qed}
\end{equation*}
\renewcommand{\qed}{}
\end{proof}

For later use, we record expressions for the operators $\widetilde{D}_C^{\mp}\widetilde{D}_C^{\pm}$.
Consider the Laplace-type operators
\begin{equation*}
\Delta^{\pm} := -\frac{\partial^2}{\partial y^2} + {\rm e}^{-2 \varphi \circ \xi^{-1}} \big(D_N^+\big)^2 \pm \big(\varphi' \circ \xi^{-1}\big) {\rm e}^{-2 \varphi \circ \xi^{-1}} D_N^+\qquad
 \text{on}\quad \Gamma^{\infty}\big(\widetilde{S}_C\big) .
\end{equation*}

\begin{Lemma}\label{lem D02}
We have $\widetilde{D}_C^{\mp}\widetilde{D}_C^{\pm} = \Delta^{\pm}$.
\end{Lemma}
\begin{proof}
By a direct computation,
\[
\widetilde{D}_C^{\mp}\widetilde{D}_C^{\pm} =
-\frac{\partial^2}{\partial y^2} + {\rm e}^{-2\varphi \circ \xi^{-1}} \big(D_N^+\big)^2 \pm \big(\varphi \circ \xi^{-1}\big)' {\rm e}^{-\varphi\circ \xi^{-1}} D_N^+.
\]
The right-hand side equals $\Delta^{\pm}$, because
\begin{equation*}
 \big(\varphi \circ \xi^{-1}\big)' = \big(\varphi' \circ \xi^{-1}\big) {\rm e}^{-\varphi \circ \xi^{-1}}.\tag*{\qed}
\end{equation*}
\renewcommand{\qed}{}
\end{proof}

\subsection{APS-boundary conditions}

Analogously to the Sobolev spaces defined in Section~\ref{sec index thm general}, we use the APS-type projection~\eqref{eq def P+} to define the following Hilbert spaces. Here we identify $\widetilde{S}_C^- \cong \widetilde{S}_C^+$ via the map $\tau$, and use the
unitary isomorphism $T$ from Lemma~\ref{lem T isom}.
\begin{align*}
&W^k_{\widetilde{D}_C}\big(\widetilde{S}_C^{\pm}\big) :=T^{-1}\big(W^k_D(S_C^{\pm})\big),
\\
&W^1_{\widetilde{D}_C}\big(\widetilde{S}_C^+; P\big) := \big\{s \in W^1_{\widetilde{D}_C}\big(\widetilde{S}_C^+\big); P^+(s|_N) = 0\big\},
\\
&W^1_{\widetilde{D}_C}\big(\widetilde{S}_C^+;1- P\big) := \big\{s \in W^1_{\widetilde{D}_C}\big(\widetilde{S}_C^+\big); (1-P^+)(s|_N) = 0\big\},
\\
&W^2_{\widetilde{D}_C}\big(\widetilde{S}_C^+; P\big) := \big\{s \in W^2_{\widetilde{D}_C}\big(\widetilde{S}_C^+\big); P^+(s|_N) = 0, (1-P^+)\big(\widetilde{D}_C^+s|_N\big) = 0\big\},
\\
&W^2_{\widetilde{D}_C}\big(\widetilde{S}_C^+; 1-P\big) := \big\{s \in W^2_{\widetilde{D}_C}\big(\widetilde{S}_C^+\big); (1-P^+)(s|_N) = 0, P^+\big(\widetilde{D}_C^-s|_N\big) = 0\big\}.
\end{align*}
\big(Analogously to~\eqref{eq def WnD}, we use continuous restriction maps $W^1_{\widetilde{D}_C}\big(\widetilde{S}_C\big) \to L^2\big(S|_N\big)$.\big)
By Lem\-ma~\ref{lem T Dpsi}, we have the bounded operators
\begin{align}
&\widetilde{D}_C^+\colon\ W^1_{\widetilde{D}_C}\big(\widetilde{S}_C^+; P\big) \to L^2\big(\widetilde{S}_C^+\big), \label{eq D0+ P}\\
&\widetilde{D}_C^-\colon\ W^1_{\widetilde{D}_C}\big(\widetilde{S}_C^+; 1-P\big) \to L^2\big(\widetilde{S}_C^+\big),\label{eq D0- 1-P} \\
&\widetilde{D}_C^- \widetilde{D}_C^+\colon\ W^2_{\widetilde{D}_C}\big(\widetilde{S}_C^+; P\big) \to L^2\big(\widetilde{S}_C^+\big), \label{eq D0-D0+ P}\\
&\widetilde{D}_C^+ \widetilde{D}_C^-\colon\ W^2_{\widetilde{D}_C}\big(\widetilde{S}_C^+; 1-P\big) \to L^2\big(\widetilde{S}_C^+\big).\label{eq D0+D0- 1-P}
\end{align}

The results in the previous subsection lead to the following conclusion, which shows that the conditions of Theorem~\ref{thm index general} are equivalent to corresponding properties of the operator $\widetilde{D}_C$ in this setting.
\begin{Proposition}\label{prop DC D0}\quad
\begin{enumerate}\itemsep=0pt
\item[$(a)$] The operator $D$ is invertible at infinity in the sense of~\eqref{eq D inv infty} if and only if there are $b_0>0$ and $u\geq a$ such that for all $s \in \Gamma^{\infty}_c\big(\widetilde{S}_C\big)$ supported in $(u, \infty)$,
\[
\big\|\widetilde{D}_Cs\big\|_{L^2(\widetilde{S}_C, \vol_{B_N} \otimes {\rm d}y)} \geq b_0 \|s\|_{L^2(\widetilde{S}_C, \vol_{B_N} \otimes {\rm d}y)}.
\]
\item[$(b)$] The operator~\eqref{eq DC+} is invertible if and only if the operator~\eqref{eq D0+ P} is.
\item[$(c)$] The operator~\eqref{eq DC-} is invertible if and only if the operator~\eqref{eq D0- 1-P} is.
\item[$(d)$] The operator $D^-_C D_C^+$ on $L^2\big(S_C^+\big)$, with domain $W^2_D\big(S_C^+; P\big)$, is self-adjoint
 if and only if the operator $\widetilde{D}_C^- \widetilde{D}_C^+$ on $L^2\big(\widetilde{S}_C^+\big)$, with domain $W^2_{\widetilde{D}_C}\big(\widetilde{S}_C^+; P\big)$ is.
\item[$(e)$] The operator
the operator $D^+_C D_C^-$ on $L^2\big(S_C^-\big)$, with domain $W^2_D\big(S_C^-; 1-P\big)$
 is self-adjoint if and only if the operator
 $\widetilde{D}_C^+ \widetilde{D}_C^-$ on $L^2\big(\widetilde{S}_C^+\big)$, with domain $W^2_{\widetilde{D}_C}\big(\widetilde{S}_C^+; 1-P\big)$
 is.
\end{enumerate}
\end{Proposition}

\begin{proof}
Part (a) follows from Lemmas~\ref{lem T isom} and~\ref{lem T Dpsi}.

For any $s \in \Gamma^{\infty}\big(\widetilde{S}_C\big)$, we have
\[
(Ts)|_N = \Phi(a) s|_N,
\]
because we view $N$ as embedded into $\overline{C}$ as $N\times \{a\}$ and into $N \times [0,\infty)$ as $N \times \{0\}$. Therefore, for any such section, $P(s|_N) = 0$ if and only if $P (Ts|_N) = 0$, and similarly for $1-P$. Hence the operator $T$, together with~\eqref{eq S0++}, defines unitary isomorphisms
\begin{align}\label{eq T W2}
\begin{split}
&T\colon\ W^1_{\widetilde{D}_C}\big(\widetilde{S}_C^+; P\big) \xrightarrow{\cong} W^1_{D}\big(S_C^+; P\big), \\
&T \circ \tau \colon\ W^1_{\widetilde{D}_C}\big(\widetilde{S}_C^+;1- P\big) \xrightarrow{\cong}W^1_{D}\big(S_C^-;1- P\big), \\
&T\colon\ W^2_{\widetilde{D}_C}\big(\widetilde{S}_C^+; P\big) \xrightarrow{\cong}W^2_{D}\big(S_C^+; P\big),\\
&T\circ \tau \colon\ W^2_{\widetilde{D}_C}\big(\widetilde{S}_C^+; 1-P\big) \xrightarrow{\cong}W^2_{D}\big(S_C^-; 1-P\big).
\end{split}
\end{align}
In the second and fourth lines, we use the fact that by definition~\eqref{eq def P} of $P$, we have
\[
P|_{L^2(S^-|_N)} = \tau_+ P^+ \tau_+^{-1},
\]
so that the isomorphisms $T \circ \tau$ preserve the given boundary conditions.

Under the isomorphisms~\eqref{eq T W2}, the pairs of operators in parts (b)--(e) correspond to each other. Here we again used Lemmas~\ref{lem T isom} and~\ref{lem T Dpsi}, and also Lemma~\ref{lem D0pm}.
\end{proof}

\subsection{Lower bounds and invertibility}

Two kinds of invertibility at infinity of $D$ are assumed in Theorem~\ref{thm index general}: there is the condition~\eqref{eq D inv infty} on $D$, and invertibility of~\eqref{eq DC+} and~\eqref{eq DC-}. To verify these conditions in the context of manifolds with $\varphi$-cusps, we use Lemma~\ref{lem bd Delta} and Proposition~\ref{prop APS cusp} below.

The proof of the following lemma is the only place where we use the condition~\eqref{eq bound phi' b} in the definition of weakly admissible cusps.
\begin{Lemma}\label{lem bd Delta}
Suppose that $M$ has weakly admissible cusps, i.e., $\varphi$ has the properties in Definition~$\ref{def weakly adm}$.
Suppose that $D_N$ is invertible. Then there are $u,b_0>0$ such that for all $s \in \Gamma_c^{\infty}\big(\widetilde{S}_C)$ supported in $N \times (u,\infty)$,
\begin{equation} \label{eq bd Delta}
\big\| \widetilde{D}_C s \big\|_{L^2} \geq b_0 \|s \|_{L^2}.
\end{equation}
\end{Lemma}
\begin{proof}
By Lemma~\ref{lem D02},
\begin{equation} \label{eq Delta lower bd}
\widetilde{D}_C^2 \geq {\rm e}^{-2 \varphi \circ \xi^{-1}} \big(D_N^2 - \big|\varphi' \circ \xi^{-1}\big| |D_N|\big)
= {\rm e}^{-2 \varphi \circ \xi^{-1}} |D_N| \bigl( |D_N| - \big|\varphi' \circ \xi^{-1}\big| \bigr).
\end{equation}

Because $D_N$ is invertible, there is a $b>0$ such that $D_N^2\geq b^2$. Because $M$ has weakly admissible cusps, there is an upper bound $\beta$ for $\varphi$, and there are
 $a'>0$ and $\alpha>0$ such that for all $x>a'$,
\[
 b-|\varphi'(x)|\geq \alpha.
\]
Now $|D_N| \geq b$, so on $N \times (\xi(a'), \infty)$,
\[
 |D_N| - \big|\varphi' \circ \xi^{-1}\big|\geq \alpha, \qquad \text{and}\qquad
{\rm e}^{-2 \varphi \circ \xi^{-1}} |D_N| \geq {\rm e}^{-2\beta} b.
\]
Hence, on $N \times (\xi(a'), \infty)$, the right-hand side of~\eqref{eq Delta lower bd} is greater than or equal to
\[
b_0^2 := \alpha b{\rm e}^{-2\beta}.
\]
Here we used that the various operators are self-adjoint and commute.
\end{proof}

\begin{Remark}
By a small adaptation of the proof of Lemma~\ref{lem bd Delta}, we can show that if $M$ has strongly admissible $\varphi$-cusps and $D_N$ is invertible, then for all $b_0>0$, there is a $u>0$ such that~\eqref{eq bd Delta} holds
 for all $s \in \Gamma_c^{\infty}\big(\widetilde{S}_C\big)$ supported in $N \times (u,\infty)$. By~\cite[Theorem~SD]{Anghel16} and Lemmas~\ref{lem T isom} and~\ref{lem T Dpsi}, this implies that $D$ has discrete spectrum. This is an analogous result to Proposition~\ref{prop discrete spec}.
\end{Remark}

\begin{Proposition}\label{prop APS cusp}
If $h$ has a positive lower bound, then the operators~\eqref{eq D0+ P} and~\eqref{eq D0- 1-P} are invertible.
\end{Proposition}

\begin{Lemma}\label{lem f g}
Let $\lambda \in \R$, $\zeta \in C^{\infty}_c(0,\infty)$ and $h \in C^{\infty}[0,\infty)$. For $u,v\geq 0$, define
\[
H_{\lambda}(u,v) := \exp\bigg( \lambda \int_u^v h(s)\, {\rm d}s \bigg).
\]
Define $f \in C^{\infty}[0,\infty)$ by
\begin{equation} \label{eq def f}
f(u) :=
\begin{cases}
\displaystyle\phantom{-}\int_0^u H_{\lambda}(u,v)\zeta(v)\, {\rm d}v & \text{if\quad $\lambda \geq 0$},
\\[4mm]
\displaystyle-\int_u^{\infty} H_{\lambda}(u,v)\zeta(v)\, {\rm d}v & \text{if\quad $\lambda <0$}.
\end{cases}
\end{equation}
Then
\begin{enumerate}\itemsep=0pt
\item[$(1)$] $f'+\lambda h f = \zeta$,
\item[$(2)$] if $\zeta = \tilde f'+\lambda h \tilde f$ for some $f \in C^{\infty}_c(0, \infty)$, then $f = \tilde f$,
\item[$(3)$] $f(0) = 0$ if $\lambda \geq 0$,
\item[$(4)$] if $h \geq \varepsilon>0$, then
$f \in L^2(0,\infty)$, and $|\lambda|\, \|f\|_{L^2} \leq \frac{2}{\varepsilon}\|\zeta\|_{L^2}$.
\end{enumerate}
\end{Lemma}
\begin{proof}
The first two points follow from computations, the third point is immediate from the definition of $f$.

For the fourth point, note that because $h \geq \varepsilon$,
\begin{equation} \label{eq est H lam}
H_{\lambda}(u, v) \leq {\rm e}^{\varepsilon \lambda(v-u)},
\end{equation}
{\samepage if either
\begin{itemize}\itemsep=0pt
\item $\lambda \geq 0$ and $v \leq u$, or
\item $\lambda<0$ and $v \geq u$.
\end{itemize}}
For $\lambda' \in \R$, define
\[
f^{\lambda'}(u) :=
\begin{cases}
\displaystyle\phantom{-}\int_0^u {\rm e}^{\lambda'(v-u)} \zeta(v)\, {\rm d}v & \text{if\quad $\lambda' \geq 0$},
\\[4mm]
\displaystyle-\int_u^{\infty} {\rm e}^{\lambda'(v-u)}\zeta(v)\, {\rm d}v & \text{if\quad $\lambda' <0$}.
\end{cases}
\]
 It is shown in the proof of Proposition 2.5 in \cite{APS1} that $f^{\lambda'} \in L^2(0,\infty)$, and $|\lambda'|\|f^{\lambda'}\|_{L^2}\leq 2\|\zeta\|_{L^2}$.

First suppose that $\lambda \geq 0$ and $\zeta \geq 0$.
Then~\eqref{eq est H lam} implies that $|f| \leq \big|f^{\varepsilon \lambda}\big|$, so $f\in L^2(0,\infty)$, and
\[
\varepsilon \lambda \|f\|_{L^2} \leq \varepsilon \lambda \big\|f^{\varepsilon \lambda}\big\|_{L^2} \leq 2\|\zeta\|_{L^2}.
\]
If $\zeta$ also takes negative values, we decompose it as a difference of two nonnegative functions, and reach the same conclusion.

Now suppose that $\lambda < 0$ and $\zeta \leq 0$. Then again,~\eqref{eq est H lam} implies that $|f| \leq \big|f^{\varepsilon \lambda}\big|$, and
\[
\varepsilon |\lambda|\, \|f\|_{L^2} \leq \varepsilon |\lambda|\, \big\|f^{\varepsilon \lambda}\big\|_{L^2} \leq 2\|\zeta\|_{L^2}.
\]
This also extends to $\zeta$ with positive values by a decomposition of $\zeta$ into nonpositive functions.
\end{proof}

For all $\lambda \in \spec\big(D_N^+\big)$, let $\big\{\varphi_{\lambda}^1, \dots, \varphi_{\lambda}^{m_{\lambda}} \big\}$ be an orthonormal basis of $\ker\big(D_N^+ - \lambda\big)$. Then
\begin{equation} \label{eq basis phi lam}
 \big\{\varphi_{\lambda}^j; {\lambda \in \spec\big(D_N^+\big)}, j=1, \dots, m_{\lambda} \big\}
 \end{equation}
 is a Hilbert basis of $L^2\big(S^+|_N\big)$ of eigensections of $D_N^+$.

\begin{proof}[Proof of Proposition~\ref{prop APS cusp}]
We prove the claim for~\eqref{eq D0+ P}, the proof for~\eqref{eq D0- 1-P} is similar.

Let $\zeta \in \Gamma^{\infty}_c\big(\widetilde{S}_C^+\big)$. Write
\begin{equation} \label{eq sum g}
\zeta = \sum_{\lambda \in \spec(D_N^+)} \sum_{j=1}^{m_{\lambda}} \zeta_{\lambda}^j \otimes \varphi_{\lambda}^j,
\end{equation}
where $\zeta_{\lambda}^j \in C^{\infty}_c(0,\infty)$. For every $\lambda \in \spec\big(D_N^+\big)$ and $j$, define the function $f_{\lambda}^j$ on $(0,\infty)$ as the function $f$ in~\eqref{eq def f}, with $\zeta$ replaced by $\zeta_{\lambda}^j$. We claim that the series
\begin{equation} \label{eq sum lambda f phi}
 \sum_{\lambda \in \spec(D_N^+)} \sum_{j=1}^{m_{\lambda}} f_{\lambda}^j \otimes \varphi_{\lambda}^j
\end{equation}
converges to an element $f \in W^1_{\widetilde{D}_C}\big(\widetilde{S}_C^+; P\big)$. We set $Q\zeta := f$ for any such $\zeta$.
Then the first point in Lemma~\ref{lem f g} implies that $\widetilde{D}_C^+ f = \zeta$. The second point in Lemma~\ref{lem f g} implies that $Q\widetilde{D}_C^+ \tilde f = \tilde f$ for any $\tilde f \in W^1_{\widetilde{D}_C}\big(\widetilde{S}_C^+; P\big)$.
The third point in Lemma~\ref{lem f g} implies that $P(f|_{N \times \{0\}}) = 0$.

To prove convergence of~\eqref{eq sum lambda f phi}, we note that by the first point in Lemma~\ref{lem f g},
\[
\big\| f_{\lambda}^j \otimes \varphi_{\lambda}^j\big\|_{W^1_{\widetilde{D}_C}}^2 = \big\|f_{\lambda}^j\big\|_{L^2}^2+\big\|\zeta_{\lambda}^j\big\|_{L^2}^2.
\]
Since $|\lambda|$ is bounded away from zero, the fourth point in Lemma~\ref{lem f g} implies that
\begin{equation} \label{eq est flam glam}
\big\| f_{\lambda}^j \otimes \varphi_{\lambda}^j\big\|_{W^1_{\widetilde{D}_C}} \leq B \big\|\zeta_{\lambda}^j\big\|_{L^2}
\end{equation}
for a constant $B>0$ independent of $\lambda$. So convergence of~\eqref{eq sum g} in $L^2\big(\widetilde{S}_C^+\big)$ implies convergence of~\eqref{eq sum lambda f phi} in $W^1_{\widetilde{D}_C}\big(\widetilde{S}_C^+\big)$.

We have just seen that \smash{$\|Q\zeta\|_{W^1_{\widetilde{D}_C}} \leq B \|\zeta\|_{L^2}$}, thus
 $Q$ extends continuously to an inverse of~\eqref{eq D0+ P}.
\end{proof}

\subsection{Adjoints}\label{sec adjoints}

\begin{Proposition}\label{prop adjoints DC}
If $h$ has a positive lower bound, then
the two operators~\eqref{eq D0+ P} and~\eqref{eq D0- 1-P} are each other's adjoints.
\end{Proposition}
\begin{proof}
We claim that for all $s_P \in W^1_{\widetilde{D}_C}\big(\widetilde{S}_C^+; P\big)$ and $s_{1-P} \in W^1_{\widetilde{D}_C}\big(\widetilde{S}_C^+; 1-P\big)$,
\begin{equation}
\label{eq adjoints DC 1}
\big(\widetilde{D}_C^+ s_P, s_{1-P}\big)_{L^2} = \big(s_P, \widetilde{D}_C^- s_{1-P}\big)_{L^2}.
\end{equation}
Indeed, suppose that
\[
s_P = s_P^N \otimes f_{P},\qquad s_{1-P} = s_{1-P}^N \otimes f_{1-P},
\]
for $s_P^N, s_{1-P}^N \in L^2\big(S^+|_N\big)$ and $f_P, f_{1-P} \in L^2[0, \infty)$ such that $s_P \in W^1_{\widetilde{D}_C}\big(\widetilde{S}_C^+; P\big)$ and $s_{1-P} \in W^1_{\widetilde{D}_C}\big(\widetilde{S}_C^+; 1-P\big)$. Then by self-adjointness of $D_N^+$ and integration by parts, the left-hand side of~\eqref{eq adjoints DC 1} equals
\begin{equation} \label{eq adjoints DC 2}
\big(s_P^N, D_N^+ s_{1-P}^N\big)_{L^2(S^+|_N)} \bigg(
\big(f_P \overline{f_{1-P}}\big)|_{0}^{\infty} + \int_0^{\infty} f_P(x)\bigl(-\overline{f_{1-P}'(x)} +h(x) \overline{f_{1-P}(x)}\bigr)\, {\rm d}x\bigg).
\end{equation}
If we further decompose the expressions with respect to eigenspaces of $D_N^+$, and use that the components of $f_P$ for positive eigenvalues equal zero at zero, and the components of $f_{1-P}$ for negative eigenvalues equal zero at zero, then we find that the components for all eigenvalues of $f_P \overline{f_{1-P}}$ are zero at zero. So the term $\big(f_P \overline{f_{1-P}}\big)|_{0}^{\infty} $ in~\eqref{eq adjoints DC 2} vanishes, and~\eqref{eq adjoints DC 2} equals the right-hand side of~\eqref{eq adjoints DC 1}.

Now if $\sigma_P := \widetilde{D}_C^+ s_P$ and $\sigma_{1-P} := \widetilde{D}_C^- s_{1-P}$, then
\eqref{eq adjoints DC 1} becomes
\[
\big( \sigma_P, \big(\widetilde{D}_C^-\big)^{-1} \sigma_{1-P}\big)_{L^2} = \big( \big(\widetilde{D}_C^+\big)^{-1}\sigma_P, \sigma_{1-P}\big)_{L^2}.
\]
The sections $\sigma_P$ and $\sigma_{1-P}$ of this type are dense in $L^2\big(\widetilde{S}_C^+\big)$. And the inverse operators $\big(\widetilde{D}_C^+\big)^{-1}$ and $\big(\widetilde{D}_C^-\big)^{-1}$ are bounded, so we find that
\[
\big(\big(\widetilde{D}_C^+\big)^{-1}\big)^* = \big(\widetilde{D}_C^-\big)^{-1}.
\]
This implies that $\big(\widetilde{D}_C^+\big)^* = \widetilde{D}_C^-$.
\end{proof}

\begin{Proposition}\label{prop DC2 sa}
If $h$ has a positive lower bound, then the operators
\eqref{eq D0-D0+ P} and~\eqref{eq D0+D0- 1-P}
are self-adjoint.
\end{Proposition}
\begin{proof}
We prove the claim for~\eqref{eq D0-D0+ P}, the proof for~\eqref{eq D0+D0- 1-P} is similar.

The operator~\eqref{eq D0+ P} is invertible by Proposition~\ref{prop APS cusp}. It maps the subspace
\[
W^2_{\widetilde{D}_C}\big(\widetilde{S}_C^+; P\big) \subset W^1\big(\widetilde{S}_C^+; P\big)
\]
onto
\[
W^1\big(\widetilde{S}_C^+; 1-P\big) \subset L^2\big(\widetilde{S}_C^+\big).
\]
 So we obtain an invertible operator
\begin{equation*} 
\widetilde{D}_C^+\colon\ W^2_{\widetilde{D}_C}\big(\widetilde{S}_C^+; P\big) \to W^1\big(\widetilde{S}_C^+; 1-P\big).
\end{equation*}
So, again by Proposition~\ref{prop APS cusp}, the composition
\begin{equation} \label{eq DCplus DCmin}
W^2_{\widetilde{D}_C}\big(\widetilde{S}_C^+; P\big) \xrightarrow{\widetilde{D}_C^+} W^1\big(\widetilde{S}_C^+; 1-P\big)\xrightarrow{\widetilde{D}_C^-} L^2\big(\widetilde{S}_C^+\big)
\end{equation}
is invertible, with bounded inverse. The adjoint of the bounded operator
\[
\big(\widetilde{D}_C^+\big)^{-1}\colon\
W^1_{\widetilde{D}_C}\big(\widetilde{S}_C^+; 1-P\big)\to W^2_{\widetilde{D}_C}\big(\widetilde{S}_C^+; P\big)
\]
is the restriction of the adjoint of
\[
\big(\widetilde{D}_C^+\big)^{-1}\colon\ L^2\big(\widetilde{S}_C^+\big)\to W^1_{\widetilde{D}_C}\big(\widetilde{S}_C^+; P\big)
\]
to $W^2_{\widetilde{D}_C}\big(\widetilde{S}_C^+; P\big)$.
By Proposition~\ref{prop adjoints DC}, this is
\[
 \big(\widetilde{D}_C^-\big)^{-1}|_{W^2_{\widetilde{D}_C}(\widetilde{S}_C^+; P) }.
\]
Again applying Proposition~\ref{prop adjoints DC}, we find that the inverse of~\eqref{eq DCplus DCmin} has adjoint
\[
W^2_{\widetilde{D}_C}\big(\widetilde{S}_C^+; P\big) \xrightarrow{ (\widetilde{D}_C^-)^{-1}} W^1_{\widetilde{D}_C}\big(\widetilde{S}_C^+; 1-P\big)\xrightarrow{(\widetilde{D}_C^+)^{-1}} L^2\big(\widetilde{S}_C^+\big).
\]
Hence, as maps from $W^2_{\widetilde{D}_C}\big(\widetilde{S}_C^+; P\big)$ to $L^2\big(\widetilde{S}_C^+\big)$,
\[
\big(\big(\widetilde{D}_C^- \widetilde{D}_C^+\big)^{-1}\big)^* = \big(\widetilde{D}_C^- \widetilde{D}_C^+\big)^{-1}.
\]
This implies that $\big(\widetilde{D}_C^- \widetilde{D}_C^+\big)^* = \widetilde{D}_C^- \widetilde{D}_C^+$.
\end{proof}
\begin{Remark}
Proposition~\ref{prop DC2 sa} can also be deduced from Proposition~\ref{prop adjoints DC} via~\cite[Theorem~X.25]{RSII}.
\end{Remark}

If $M$ has weakly admissible $\varphi$-cusps, then $\varphi$ has an upper bound, so $h = {\rm e}^{-\varphi \circ \xi}$ has a positive lower bound. Therefore,
Proposition~\ref{prop ind cusps cond} follows from
Proposition~\ref{prop DC D0}, Lemma~\ref{lem bd Delta} and Propositions~\ref{prop APS cusp} and~\ref{prop DC2 sa}.

\subsection{Cusp contributions}\label{sec cusp contr}

To prove Proposition~\ref{prop ind cusp contr}, we express the contribution from infinity in Theorem~\ref{thm index general} in terms of the operator $\widetilde{D}_C$. Recall the definition of the function $\Phi$ in~\eqref{eq def Phi}.
\begin{Lemma}\label{lem lambda lambda tilde}
Suppose that the operator~\eqref{eq D0-D0+ P} is self-adjoint, and
let $\tilde \lambda_s^{P}$ be the Schwartz kernel of ${\rm e}_P^{-s \widetilde{D}_C^- \widetilde{D}_C^+} \widetilde{D}_C^-$. Then
for all $n,n' \in N$, $x,x' \in (a, \infty)$ and $s>0$,
\[
\lambda_s^{P}(n, x; n', x') = {\rm e}^{(1-p)\varphi(x')} \frac{\Phi(x)}{\Phi(x')} \tilde \lambda_s^{P}(n, \xi(x); n', \xi(x')).
\]
\end{Lemma}
\begin{proof}
The kernel $ \lambda_s^{P}$ is defined with respect to the Riemannian density ${\rm e}^{p\varphi}\, {\rm d}n \, {\rm d}x$ on $C$, whereas~$\tilde \lambda_s^{P}$ is defined with respect to the Riemannian density $dn \, {\rm d}y$ on $N \times (0, \infty)$. Furthermore, Lemma~\ref{lem T Dpsi} and the third isomorphism on~\eqref{eq T W2} imply that
\[
{\rm e}_P^{-s D_C^- D_C^+}D_C^- =
 T \circ {\rm e}_P^{-s \widetilde{D}_C^- \widetilde{D}_C^+} \widetilde{D}_C^-\circ T^{-1}.
\]
We find that for all $n \in N$ and $x' \in (a,\infty)$ and $s \in \Gamma^{\infty}_c(S_C)$,
\begin{align*}
\begin{split}
&\int_N \int_a^{\infty} \lambda_s^{P}(n, x; n', x') s(n', x') {\rm e}^{p\varphi(x')}\, {\rm d}x' \, {\rm d}n' \\
&\qquad{}= \bigl({\rm e}_P^{-s D_C^- D_C^+}D_C^-s \bigr)(n, x)
= T \bigl({\rm e}_P^{-s \widetilde{D}_C^- \widetilde{D}_C^+} \widetilde{D}_C^- T^{-1}s \bigr)(n, x)
\\
&\qquad{}= \Phi(x) \int_N \int_0^{\infty} \tilde \lambda_s^{P}(n, \xi(x); n', y') \frac{1}{\Phi(\xi^{-1}(y'))} s(n', \xi^{-1}(y'))
\, {\rm d}n'\, {\rm d}y'.
\end{split}
\end{align*}
By a substitution $x' = \xi^{-1}(y')$, the latter integral equals
\[
 \Phi(x) \int_N \int_a^{\infty} \tilde \lambda_s^{P}(n, \xi(x); n', \xi(x')) \frac{1}{\Phi(x')} s(n', x') {\rm e}^{\varphi(x')}
\, {\rm d}n'\, {\rm d}x'.
\]
Here we used that $\xi' = {\rm e}^{\varphi}$.
\end{proof}

 Recall the choice of the Hilbert basis~\eqref{eq basis phi lam} of $L^2\big(S^+|_N\big)$ of eigensections of $D_N^+$. Let $\rho^{\lambda, \pm}$ and $\theta^{\lambda, \pm}_{\nu}$ be as $\rho$ and $\theta_{\nu}$ in Theorem~\ref{thm spec decomp}, with $q = q^{\pm}_{\lambda}$ as in~\eqref{eq def q lambda}.
\begin{Lemma}\label{lem kernel APS}
For all $s>0$, the Schwartz kernel $\tilde \lambda_s^{P}$ in Lemma~$\ref{lem lambda lambda tilde}$ equals
\begin{align}
\sum_{\lambda >0} \sum_{j=1}^{m_{\lambda}}&\bigg(\int_{\R}
{\rm e}^{-s\nu} \theta^{\lambda, +}_{\nu}\otimes \bigg( \frac{\rm d}{{\rm d}y}+\lambda {\rm e}^{-\varphi \circ \xi^{-1}}\bigg) \theta^{\lambda, +}_{\nu} {\rm d}\rho^{\lambda, +}(\nu) \bigg)\otimes \big(\varphi_{\lambda}^j \otimes \varphi_{\lambda}^j \big)\nonumber
\\
&+ \sum_{\lambda <0}
 \sum_{j=1}^{m_{\lambda}}
\bigg(\int_{\R}
{\rm e}^{-s\nu} \biggl( - \frac{\rm d}{{\rm d}y}+\lambda {\rm e}^{-\varphi \circ \xi^{-1}}\biggr)\theta^{\lambda, -}_{\nu}\otimes \theta^{\lambda, -}_{\nu} {\rm d}\rho^{\lambda, -}(\nu) \bigg)\otimes \big(\varphi_{\lambda}^j \otimes \varphi_{\lambda}^j \big).
\label{eq kernel cusp contr}
\end{align}
Here we identify $S^+_N \cong \big(S^+_N\big)^*$ using the metric, so we view $\varphi_{\lambda}^j \otimes \varphi_{\lambda}^j$ as a section of $S^+_N \boxtimes \big(S^+_N\big)^*$.
\end{Lemma}
\begin{proof}
We extend the projection~\eqref{eq def P+} to a projection
\[
P \colon\ L^2\big(\widetilde{S}_C^+\big) \cong L^2\big(S^+|_N\big) \otimes L^2(0,\infty)
 \xrightarrow{P \otimes 1} L^2(S^+|_N)_{>0} \otimes L^2 (0,\infty)\hookrightarrow L^2\big(\widetilde{S}_C^+\big).
\]
Then \cite[Proposition~3.5]{HW21a} states that
\begin{equation} \label{eq decomp eP eF}
{\rm e}_P^{-s\widetilde{D}_C^-\widetilde{D}_C^+}\widetilde{D}_C^- = {\rm e}_F^{-s\widetilde{D}_C^-\widetilde{D}_C^+}\widetilde{D}_C^- P + \widetilde{D}_C^-{\rm e}_F^{-s\widetilde{D}_C^+\widetilde{D}_C^-}(1-P).
\end{equation}

 For $s>0$, let ${\rm e}_F^{-s \widetilde{D}_C^{\mp} \widetilde{D}_C^{\pm}}$ be the heat operator for the Friedrichs extension of
\[
\widetilde{D}_C^{\mp} \widetilde{D}_C^{\pm}\colon\ \Gamma^{\infty}_c\big(\widetilde{S}_C^+\big) \to L^2\big(\widetilde{S}_C^+\big).
\]
Let $\kappa_s^{F, \pm}$ be its Schwartz kernel.
By~\eqref{eq Delta lambda},~\eqref{eq def q lambda} and Lemma~\ref{lem D02}, the restriction of $\widetilde{D}_C^{\mp} \widetilde{D}_C^{\pm}$ to $\ker\big(D_N^+ - \lambda\big) \otimes L^2(0,\infty)$ equals $\Delta_{q_{\lambda}^{\pm}}$. So by Theorem~\ref{thm spec decomp}, the Schwartz kernel $\kappa_{\lambda, s}^{F, \pm}$ of the restriction of ${\rm e}_F^{-s \widetilde{D}_C^{\mp} \widetilde{D}_C^{\pm}}$ to $\ker\big(D_N^+ - \lambda\big) \otimes L^2(0, \infty)$ is
\[
\kappa_{\lambda, s}^{F, \pm} =
\sum_{j=1}^{m_{\lambda}}\bigg(\int_{\R}{\rm e}^{-s\nu} \theta^{\lambda, \pm}_{\nu}\otimes \theta^{\lambda, \pm}_{\nu} \, {\rm d}\rho^{\lambda, \pm}(\nu)\bigg) \otimes \big(\varphi_{\lambda}^j \otimes \varphi_{\lambda}^j \big).
\]
\big(Note that $\sum_{j=1}^{m_{\lambda}} \varphi_{\lambda}^j \otimes \varphi_{\lambda}^j$ is the identity operator on $\ker\big(D_N^+ - \lambda\big)$.\big)
The Schwartz kernel of $\widetilde{D}_C^- {\rm e}_F^{-s \widetilde{D}_C^{+}\widetilde{D}_C^-}$ equals $\widetilde{D}_C^-$ applied to the first entry of $\kappa_s^{F, -}$, so its restriction to $\ker\big(D_N^+ - \lambda\big) \otimes L^2([0, \infty))$ is
\begin{equation} \label{eq D kappa F lambda}
\sum_{j=1}^{m_{\lambda}}\bigg(\int_{\R}
{\rm e}^{-s\nu} \biggl( - \frac{\rm d}{{\rm d}y}+\lambda {\rm e}^{-\varphi \circ \xi^{-1}}\biggr) \theta^{\lambda, -}_{\nu}\otimes \theta^{\lambda, -}_{\nu} \, {\rm d}\rho^{\lambda, -}(\nu)\bigg) \otimes \big(\varphi_{\lambda}^j \otimes \varphi_{\lambda}^j \big).
\end{equation}
The Schwartz kernel of $ {\rm e}_F^{-s \widetilde{D}_C^- \widetilde{D}_C^{+} }\widetilde{D}_C^{-}$ equals the adjoint $\widetilde{D}_C^+$ of $\widetilde{D}_C^-$ applied to the second entry of $\kappa_s^{F, +}$, so its restriction to $\ker\big(D_N^+ - \lambda\big) \otimes L^2(0, \infty)$ is
\begin{equation} \label{eq kappa F D lambda}
\sum_{j=1}^{m_{\lambda}}\bigg(\int_{\R}
{\rm e}^{-s\nu} \theta^{\lambda, +}_{\nu}\otimes \bigg( \frac{\rm d}{{\rm d}y}+\lambda {\rm e}^{-\varphi \circ \xi^{-1}}\bigg) \theta^{\lambda, +}_{\nu} \, {\rm d}\rho^{\lambda, +}(\nu)\bigg) \otimes \big(\varphi_{\lambda}^j \otimes \varphi_{\lambda}^j \big).
\end{equation}
The claim follows from~\eqref{eq decomp eP eF} and the expressions~\eqref{eq D kappa F lambda} and~\eqref{eq kappa F D lambda} for the relevant Schwartz kernels on eigenspaces of $D_N^+$.
\end{proof}

\begin{Lemma}\label{lem kernel on diagonal}
In the situation of Lemma~$\ref{lem lambda lambda tilde}$, we have for all $a'>a$,
\begin{gather}
\int_N \tr\big( g \lambda_s^{P}\big(g^{-1}n, a'; n, a'\big) \big) {\rm d}n
 = {\rm e}^{(1-p)\varphi(a')}\sum_{\lambda \in \spec(D_N^+)}\sgn(\lambda)\tr(g|_{\ker(D_N^+ - \lambda)}) \nonumber
 \\ \qquad
 {}\times
\int_{\R}
{\rm e}^{-s\nu} \theta^{|\lambda|, +}_{\nu} (\xi(a')) \bigg( \bigg( \frac{\rm d}{{\rm d}y}+|\lambda| {\rm e}^{-\varphi \circ \xi^{-1}}\bigg) \theta^{|\lambda|, +}_{\nu} \bigg)(\xi(a'))\,{\rm d}\rho^{|\lambda|, +}(\nu) .
\label{eq kernel on diagonal}
\end{gather}
\end{Lemma}

\begin{proof}
It follows directly from~\eqref{eq def q lambda} that for all $\lambda \in \R$,
\[
q^{+}_{\lambda} = q^{-}_{-\lambda}.
\]
This implies that, for all $\lambda \in \R$ and $\nu \in \C$, with notation as in~\eqref{eq Delta lambda} and below,
\[
\Delta^+_{\lambda}= \Delta^-_{-\lambda},\qquad
\theta^{\lambda, +}_{\nu} = \theta^{-\lambda,-}_{\nu},\qquad
\rho^{\lambda, +} = \rho^{-\lambda, -}.
\]
The last two relations imply in particular that for all $\lambda<0$,
\begin{equation*} 
\theta^{\lambda, -}_{\nu} = \theta^{|\lambda|, +}_{\nu},\qquad
\rho^{\lambda, -} = \rho^{|\lambda|, +},\qquad
 -\frac{\rm d}{{\rm d}y}+\lambda {\rm e}^{-\varphi \circ \xi^{-1}}= \sgn(\lambda) \bigg( \frac{\rm d}{{\rm d}y}+|\lambda| {\rm e}^{-\varphi \circ \xi^{-1}}\bigg).
\end{equation*}
These equalities, together with Lemmas~\ref{lem lambda lambda tilde} and~\ref{lem kernel APS} imply that for all $n \in N$,
\begin{align*}
g \lambda_s^{P}\big(g^{-1}n, a'; n, a'\big) = {\rm e}^{(1-p)\varphi(a')}
\sum_{\lambda >0} \sum_{j=1}^{m_{\lambda}} &\sgn(\lambda)
\bigg(\int_{\R}
{\rm e}^{-s\nu} \theta^{|\lambda|, +}_{\nu} (\xi(a')) \bigg( \frac{\rm d}{{\rm d}y}+|\lambda| {\rm e}^{-\varphi \circ \xi^{-1}}\bigg)
\\
& \times\theta^{|\lambda|, +}_{\nu} (\xi(a'))\, {\rm d}\rho^{\lambda, +}(\nu) \bigg)
\big(g \varphi_{\lambda}^j (g^{-1}n) \otimes \varphi_{\lambda}^j (n) \big).
\end{align*}
This equality and
\[
\sum_{j=1}^{m_{\lambda}}
\int_{N} \tr\big(g\varphi_{\lambda}^j\big(g^{-1}n\big) \otimes \varphi_{\lambda}^j(n)\big) \, {\rm d}n =
\sum_{j=1}^{m_{\lambda}} \big(g\cdot \varphi_{\lambda}^j, \varphi_{\lambda}^j\big)_{L^2} = \tr\big(g|_{\ker(D_N^+ - \lambda)}\big)
\]
together
imply~\eqref{eq kernel on diagonal}.
\end{proof}

\begin{proof}[Proof of Proposition~\ref{prop ind cusp contr}.]
By Proposition~\ref{prop Dphi}, the function $f_1$ in~\eqref{eq D on U} equals ${\rm e}^{-\varphi}$ in our situation. So it
 follows from Lemma~\ref{lem kernel on diagonal} that for all $a'>a$
 \begin{align}
 A_{g}(D_C, a') = - &{\rm e}^{-p \varphi(a')} \int_{0}^{\infty}
\sum_{\lambda \in \spec(D_N^+)}\sgn(\lambda)\tr\big(g|_{\ker(D_N^+ - \lambda)}\big)\nonumber
\\
&\times\int_{\R}{\rm e}^{-s\nu} \theta^{|\lambda|, +}_{\nu} (\xi(a')) \bigg( \bigg( \frac{\rm d}{{\rm d}y}+|\lambda| {\rm e}^{-\varphi \circ \xi^{-1}}\bigg) \theta^{|\lambda|, +}_{\nu} \bigg)(\xi(a'))\,{\rm d}\rho^{|\lambda|, +}(\nu) \, {\rm d}s\nonumber
 \\
 =-&\frac{1}{2}\eta_g^{\varphi}\big(D_N^+, a'\big).
\label{eq comp Ag}
 \end{align}
 In particular, because $A_{g}(D_C, a')$ converges by Theorem~\ref{thm index general}, so does $\eta_g^{\varphi}\big(D_N^+, a'\big)$.

Vanishing of $\eta^{\varphi}\big(D_N^+, a'\big)$ when $D_N^+$ has $g$-symmetric spectrum around zero follows directly from the definition~\eqref{eq cusp cpt}: then the term corresponding to $\lambda \in \spec(D_N)$ equals minus the term corresponding to $-\lambda$.
\end{proof}

Proposition~\ref{prop discrete spec} was proved at the start of this section, and Proposition~\ref{prop ind cusps cond} was proved at the end of Section~\ref{sec adjoints}. So Propositions~\ref{prop ind cusps cond}--\ref{prop discrete spec} are proved, and the proof of Theorem~\ref{thm index cusps} is complete.

\section{Cylinders}

If $\varphi = 0$, then the metric~\eqref{eq g phi} is the cylinder metric $B_N + {\rm d}x^2$. We show that the cusp contribution $\eta^0_g\big(D_N^+, a'\big)$ then equals Donnelly's $g$-delocalised version of the Atiyah--Patodi--Singer $\eta$-invariant, for all $a'>a$. The computation in this subsection is a spectral counterpart of the geometric computation in~\cite[Section~4]{HW21a}.

We start by recalling~\cite[Proposition~5.1]{HW21a}.
\begin{Proposition}\label{prop vanish}
Let $(\lambda_j)_{j=1}^{\infty}$ and $(a_j)_{j=1}^{\infty}$ be sequences in $\R$ such that $|\lambda_1|>0$, and $|\lambda_j| \leq |\lambda_{j+1}|$ for all $j$, and such that
there are $c_1,c_2, c_3, c_4>0$ such that for all $j$,
\begin{equation*} 
|\lambda_j| \geq c_1 j^{c_2},\qquad |a_j| \leq c_3 j^{c_4}.
\end{equation*}
Then for all $a'>0$,
\begin{equation*}
\int_0^{\infty} \sum_{j=1}^{\infty} \sgn(\lambda_j) a_j \frac{{\rm e}^{-\lambda_j^2 s} {\rm e}^{-a'^2/s}}{\sqrt{s}} \bigg(\frac{a'}{s} - |\lambda_j| \bigg) \, {\rm d}s = 0.
\end{equation*}
\end{Proposition}

For a function $f \in L^1(\R)$, we write
\[
\check{f}(x) := \int_{\R} {\rm e}^{{\rm i}x\zeta} f(\zeta)\, {\rm d}\zeta
\]
for its inverse Fourier transform (up to a possible power of $2\pi$).
\begin{Lemma}\label{lem IFT}
Let $f \in L^1(\R)$, and $\alpha,\beta \in \R$.
If $f$ is even, then
\[
\int_0^{\infty}\sin(\alpha\mu)\sin(\beta\mu)f(\mu)\,{\rm d}\mu=\frac14 \bigl( -\check{f}(\alpha+\beta)+\check{f}(\alpha-\beta) \bigr).
\]
If $f$ is odd, then
\[
\int_0^{\infty}\sin(\alpha\mu)\cos(\beta\mu)f(\mu)\,{\rm d}\mu=\frac1{4{\rm i}}\bigl( \check{f}(\alpha+\beta)+\check{f}(\alpha-\beta)\bigr).
\]
\end{Lemma}

\begin{Lemma}\label{lem rho cyl}
If $\varphi = 0$, then the spectral measure ${\rm d}\rho^{|\lambda|, +}$ in~\eqref{eq cusp cpt} equals
\[
{\rm d}\rho^{|\lambda|, +}(\nu)
=\begin{cases}\dfrac{1}{\pi}\sqrt{\nu-\lambda^2}\,{\rm d}\nu &\text{if\quad $\nu\ge\lambda^2$},
\\
0 & \text{if\quad $\nu<\lambda^2$}. \end{cases}
\]
\end{Lemma}

\begin{proof}The proof is analogous to the computation in~\cite[Section~4.1]{Titchmarsh62} in the Neumann case.

With notation as in Proposition~\ref{prop rho Titch}, we now have
\[
\theta_1(x, \nu)= \frac{1}{\mu}\sin(\mu x),\qquad
\theta_2(x, \nu)= - \cos(\mu x),
\]
with $\mu := \sqrt{\nu - \lambda^2}$. If $\nu$ has positive imaginary part, then we choose the square root with positive real and imaginary parts. Then
 $f(\nu) = -i\mu$ has negative imaginary part, and
 \[
 \theta_2(x, \nu) + f(\nu) \theta_1(x, \nu) = -{\rm e}^{{\rm i}\mu x}
 \]
 defines a function in $L^2([0,\infty))$ if $\Imag(\nu)>0$.
 With our choice of square roots, we have for all $\nu \in \R$,
 \[
 \lim_{\delta \downarrow 0} -\Imag(f(\nu+{\rm i}\delta)) =
 \begin{cases}
 \mu & \text{if\quad $\nu \geq \lambda^2$},\\
 0 & \text{if\quad$\nu < \lambda^2$}.
 \end{cases}
 \]
 This implies the claim via
 Proposition~\ref{prop rho Titch}.
\end{proof}

The \emph{$g$-delocalised $\eta$-invariant} of $D_N^+$ \cite{Donnelly, HWW, Lott99} is
\[
\eta_g\big(D_N^+\big) = \frac{1}{\sqrt{\pi}}\int_0^{\infty} \Tr\big(g D_N^+ {\rm e}^{-s(D_N^+)^2}\big) \frac{1}{\sqrt{s}}\, {\rm d}s.
\]
If $g = e$, this equals the classical $\eta$-invariant of $D_N^+$.
\begin{Proposition}\label{prop cusp contr cylinder}
If $\varphi = 0$, then for all $a'>a$
\[
\eta^{0}_g\big(D_N^+, a'\big)=\eta_g\big(D_N^+\big).
\]
\end{Proposition}

\begin{proof}We apply the definition~\eqref{eq cusp cpt} of cusp contributions with $\varphi = 0$.
We look for solutions of $-\theta''+\lambda^2=\nu\theta $ satisfying
\[
\theta(0,\nu)=0,\qquad \theta'(0,\nu)=1,
\]
 and find the eigenfunctions $\theta^{|\lambda|,+}_{\nu}(y) =\frac{1}{\sqrt{\nu-\lambda^2}}\sin\big(\sqrt{\nu-\lambda^2}y\big)$.
 Let $a'>a$, and set $a'':= {a'-a = \xi(a')}$. Then
 by Lemma~\ref{lem rho cyl},~\eqref{eq cusp cpt} becomes
\begin{align}
\eta^{0}_g\big(D_N^+, a'\big) ={}&\frac{2}{\pi} \int_{0}^{\infty}
\sum_{\lambda \in \spec(D_N^+)} \sgn(\lambda)\tr\big(g|_{\ker(D_N^+ - \lambda)}\big)
\int_{\lambda^2}^{\infty}
{\rm e}^{-s\nu}
\frac{\sin\big(\sqrt{\nu-\lambda^2}a''\big)}{\sqrt{\nu-\lambda^2}}\nonumber
\\
 &\times\bigg( \cos\big(\sqrt{\nu-\lambda^2}a''\big) +|\lambda| \frac{\sin\big(\sqrt{\nu-\lambda^2}a''\big)}{\sqrt{\nu-\lambda^2}}
 \bigg)\sqrt{\nu-\lambda^2}\, {\rm d}\nu \, {\rm d}s.
 \label{eq cusp contr cyl 1}
 \end{align}

The change of variables $\mu=\sqrt{\nu-\lambda^2}$ and ${\rm d}\nu=2\sqrt{\nu-\lambda^2}\,{\rm d}\mu$ reduces it to
\begin{align*}
\begin{split}
\eta^{0}_g\big(D_N^+, a'\big) ={}&\frac{4}{\pi} \int_{0}^{\infty}
\sum_{\lambda \in \spec(D_N^+)} \sgn(\lambda)\tr\big(g|_{\ker(D_N^+ - \lambda)}\big)
\\
&\times\int_{0}^{\infty}{\rm e}^{-s(\mu^2+\lambda^2)}\sin(\mu a'')
 \big( \mu\cos(\mu a'') +|\lambda| \sin(\mu a'') \big)\,{\rm d}\mu \, {\rm d}s.
 \end{split}
\end{align*}

Applying Lemma~\ref{lem IFT} and using $f(\mu)={\rm e}^{-s\mu^2}$, so $\check f(x)=\sqrt{\frac{\pi}{s}}{\rm e}^{-\frac{x^2}{4s}}$ and $(\mu \mapsto \mu f(\mu))^{\vee}=\frac{1}{i}(\check{f})'$, we have
\begin{align*}
&\int_0^{\infty}\sin^2(\mu a''){\rm e}^{-s\mu^2}\, {\rm d}\mu =\frac{\sqrt{\pi}}{4\sqrt{s}}\big(1-{\rm e}^{-\frac{a''^2}{s}}\big),
\\
&\int_0^{\infty}\sin(\mu a'')\cos(\mu a'')\mu {\rm e}^{-s\mu^2}\, {\rm d}\mu=\frac{\sqrt{\pi} a''}{4s^{\frac32}}{\rm e}^{-\frac{a''^2}{s}}.
\end{align*}

Therefore,
\begin{align}
\eta^{0}_g\big(D_N^+, a'\big)={}&\frac{1}{\sqrt{\pi}}\! \int_{0}^{\infty}\!\!\!\!\!
\sum_{\lambda \in \spec(D_N^+)}\!\!\! \sgn(\lambda)\tr\big(g|_{\ker(D_N^+ - \lambda)}\big)
{\rm e}^{-s\lambda^2}\!
 \bigg( \frac{a''}{s^{\frac32}}{\rm e}^{-\frac{a''^2}{s}}+\frac{|\lambda|}{\sqrt{ s}}-\frac{|\lambda|}{\sqrt{ s}}{\rm e}^{-\frac{a''^2}{s}} \bigg) \, {\rm d}s\nonumber
 \\
={}& \frac{1}{\sqrt{\pi}}\int_{0}^{\infty}
\sum_{\lambda\in\spec(D_N^+)} \tr\big(g|_{\ker(D_N^+ - \lambda)}\big) \frac{{\rm e}^{-s\lambda^2}\lambda}{\sqrt{s}}\, {\rm d}s \nonumber
\\
& -\frac{1}{\sqrt{\pi}} \int_0^{\infty} \sum_{\lambda\in\spec(D_N^+)}\sgn(\lambda)\tr\big(g|_{\ker(D_N^+ - \lambda)}\big)\frac{{\rm e}^{-s\lambda^2}{\rm e}^{-\frac{a''^2}{s}}}{\sqrt{s}}\bigg(\frac{a''}{s}-|\lambda|\bigg)\, {\rm d}s. \!\!\label{eq eta cylinder}
\end{align}
The first term equals $\eta_g\big(D_N^+\big)$. For the second term, we
use Proposition~\ref{prop vanish},
and take $\lambda_j$ to be the $j$th eigenvalue of $D_N^+$ (ordered by absolute values), and $a_j := \tr\big(g|_{\ker(D_N^+ - \lambda_j)}\big)$.
Then Weyl's law for $D_N^+$ shows that $\lambda_j$ has the growth behaviour assumed in the proposition. And
\[
|a_j| \leq \dim\big(\ker\big(D_N^+ - \lambda_j\big)\big)
\]
grows at most polynomially by Weyl's law.
Hence Proposition~\ref{prop vanish} applies, and implies that the second term in~\eqref{eq eta cylinder} is zero.
\end{proof}

\begin{Remark}\label{rem cyl zero at a}
By Proposition~\ref{prop cusp contr cylinder}, the cusp contribution $\eta^{0}_g\big(D_N^+, a'\big)$ is independent of $a'>a$ in this case. Furthermore, we see directly from~\eqref{eq cusp contr cyl 1} that a version of $\eta^{0}_g\big(D_N^+, a'\big)$ with $a'$ replaced by $a$ equals zero. This illustrates the fact that the limit on the right-hand side of~\eqref{eq index cusps 2} does not equal the expression~\eqref{eq cusp cpt} with $a'$ replaced by~$a$.
\end{Remark}

As a consequence of Theorem~\ref{thm index cusps} and Proposition~\ref{prop cusp contr cylinder}, we obtain Donnelly's equivariant APS index theorem \cite{Donnelly} for the index of $D|_Z$ with APS boundary conditions at $N$:
\[
\ind_G^{\APS}(D|_Z)(g) = \int_{Z^g} \AS_g(D) -\frac{1}{2}\eta_g\big(D_N^+\big).
\]
Indeed, if $D_N^+$ is invertible, then the left-hand side equals $\ind_G(D)(g)$. In general, one replaces Theorem~\ref{thm index cusps} by Theorem~\ref{thm index non invtble} below.

\section[Non-invertible D\_N]{Non-invertible $\boldsymbol{D_N}$}\label{sec DN not invtble}\label{sec DN not inv}

In this section, we do not assume that $D_N$ is invertible. Because $N$ is compact, $D_N^+$ has discrete spectrum. Let $\varepsilon>0$ be such that
\begin{equation} \label{eq prop eps}
\spec\big(D_N^+\big) \cap (-2\varepsilon, 2\varepsilon) \subset \{0\}.
\end{equation}
Let $w \in C^{\infty}(M)$ be a function such that for all $x\geq a$ and all $n \in N$,
\begin{equation} \label{eq prop w}
w(n, x) = x.
\end{equation}
Consider the operator
\begin{equation} \label{eq def Depsw}
D^{\varepsilon w} := {\rm e}^{-\varepsilon w} D {\rm e}^{\varepsilon w}.
\end{equation}
This operator
 equals~\eqref{eq Dphi} on $C$, with $D_N^+$ replaced by the invertible operator $D_N^+ +\varepsilon$. Therefore, much of the proof of Theorem~\ref{thm index cusps} applies to $D^{\varepsilon w}$, apart from the fact that this operator and $D_N^+ + \varepsilon$ are not Dirac operators of the form~\eqref{eq Dirac}. This affects the limits as $t \downarrow 0$ of heat operators associated to these operators.

In the proof of Theorem~\ref{thm index general}, given in \cite{HW21a}, heat kernel asymptotics were used that may not apply to $D_N^+ + \varepsilon$. We therefore start from a version of this theorem where the hard cutoff between $Z^g \cup \big(N^g \times (a,a']\big)$ and the contribution from infinity $A_g(D_C, a')$
 is replaced by a smooth cutoff function. Let $\psi \in C^{\infty}(M)$ be such that
 \begin{equation} \label{eq prop psi}
\psi|_Z \equiv 1,\qquad \psi|_{N \times [a+1, \infty)}\equiv 0.
 \end{equation}
 For $t>0$, define
 \[
 A_g^t(D_C, \psi) := \int_0^{\infty}\!\!\!
\int_N \int_a^{\infty} \tr\big(g \lambda_s^{P}\big(g^{-1}n, x; n, x\big)\big) \psi'(x) f_1(x)\, {\rm d}x\, {\rm d}n\, {\rm d}s,
 \]
 and
 \begin{align*} 
\eta^{\varphi, t}_g(D_N^+, \psi) =&-2 \int_{t}^{\infty}
\int_{a}^{\infty} \psi' ( x){\rm e}^{-p\varphi(x)}
\sum_{\lambda \in \spec(D_N^+)}\sgn(\lambda) \tr\big(g|_{\ker(D_N^+ - \lambda)}\big)
\\
&\times\!\int_{\R}\!{\rm e}^{-s\nu}
\theta^{|\lambda|, +}_{\nu} (\xi(x))
 \big( \big(\theta^{|\lambda|, +}_{\nu} \big)'(\xi(x)) +|\lambda| {\rm e}^{-\varphi(x)} \theta^{|\lambda|, +}_{\nu} (\xi(x))
 \big)\,{\rm d}\rho^{|\lambda|, +}(\nu)\, {\rm d}x \, {\rm d}s.
 \end{align*}

We use the following standard regularisation method.
\begin{Definition}
For a function $f(t)$ that has an asymptotic expansion in $t$ as $t \downarrow 0$, the \emph{regularised limit} $\LIM_{t \downarrow 0}f(t)$ is
the coefficient of $t^0$ in such an asymptotic expansion.

The \emph{regularised $g$-delocalised $\varphi$-cusp contribution} associated to $D_N^+ + \varepsilon$ and $\psi$ is
\[
\eta_g^{\varphi, \reg}\big(D_N^+ + \varepsilon, \psi\big) := \LIM_{t \downarrow 0} \eta_g^{\varphi, t}\big(D_N^+ + \varepsilon, \psi\big).
\]
\end{Definition}

 The condition~\eqref{eq bound phi' b} with $b$ replaced by $\varepsilon$ implies that $D^{\varepsilon w}$ is Fredholm, via Lemma~\ref{lem bd Delta}.
Theorem 4.14 in \cite{HW21a} then becomes
 \begin{gather} \label{eq index general not inv}
 \ind_G(D^{\varepsilon w})(g) = \LIM_{t\downarrow 0} \bigl(
 \Tr\big(g\circ {\rm e}^{-t\tilde D^-_{\varepsilon w} \tilde D^+_{\varepsilon w}}\psi\big)\! - \Tr\big(g\circ {\rm e}^{-t\tilde D^+_{\varepsilon w} \tilde D^-_{\varepsilon w}}\psi\big)\! + A_{g}^t\big(D_C^{\varepsilon w}, \psi\big) \bigr).\!\!\!\!
 \end{gather}
 Here $\tilde D_{\varepsilon w}$ is an extension of $D^{\varepsilon w}$ to a closed manifold containing $M \setminus (N \times (a+1, \infty))$,
and we used the fact that the left-hand side is independent of $t$. The proof of~\eqref{eq index general not inv} is a direct analogy of the proof of Theorem 4.14 in \cite{HW21a}.
 \begin{Example}
 Suppose that $N$ is the circle, and $D_N^+ = {\rm i}\frac{{\rm d}}{{\rm d}\theta}$. Then $D_N^+ + 1/2$ is invertible and has symmetric spectrum.
 So
 \[
 A_{e}^t\big(D_C^{\varepsilon w}, \psi\big) = \eta^{\varphi, t}_{e}\big(D_N^++1/2, \psi\big) = 0
 \]
 for all $t$. Hence
~\eqref{eq index general not inv} becomes
 \[
 \ind\big(D^{w/2}\big) = \LIM_{t\downarrow 0} \bigl(
 \Tr\big({\rm e}^{-t\tilde D^-_{ w/2} \tilde D^+_{w/2}}\psi\big) - \Tr\big({\rm e}^{-t\tilde D^+_{w/2} \tilde D^-_{w/2}}\psi\big) \bigr).
 \]
 \end{Example}

\begin{Theorem}\label{thm index non invtble}
Suppose that $M$ has weakly admissible $\varphi$-cusps, where~\eqref{eq bound phi' b} holds on an interval $(a', \infty)$, with $a'>a$, and $b$ replaced by $\varepsilon$.
Then $D^{\varepsilon w}$ is Fredholm, and its index is independent of $\varepsilon$ and $w$ with the properties mentioned.
And for $\psi \in C^{\infty}(M)$ satisfying~\eqref{eq prop psi},
 \begin{equation} \label{eq index non invtble}
\ind_G\big(D^{\varepsilon w}\big)(g) = \int_{M^g} \psi|_{M^g} \AS_g(D)
-\frac{1}{2} \lim_{\varepsilon \downarrow 0}\eta^{\varphi, \reg}_g\big(D_N^+ +\varepsilon, \psi\big).
\end{equation}
\end{Theorem}

\begin{proof}
As noted above~\eqref{eq index general not inv}, the operator $D^{\varepsilon w}$ is Fredholm if $M$ has weakly admissible $\varphi$-cusps with respect to the spectral gap $2\varepsilon$ of this operator.
If $\varepsilon'>0$ has the same property~\eqref{eq prop eps} as $\varepsilon$, then
\[
D^{\varepsilon w} - D^{\varepsilon' w} = (\varepsilon' - \varepsilon)c({\rm d}w).
\]
This is a bounded vector bundle endomorphism, so the linear path between $D^{\varepsilon w}$ and $D^{\varepsilon' w}$ is continuous. And all operators on this path are Fredholm, so $\ind(D^{\varepsilon w}) =\ind(D^{\varepsilon' w})$.

If $w' \in C^{\infty}(M)$ has the same property~\eqref{eq prop w} as $w$, then
\[
D^{\varepsilon w} - D^{\varepsilon w'} = - \varepsilon c({\rm d}(w-w')).
\]
Because $w-w' = 0$ outside a compact set, $D^{\varepsilon w'}$ is a compact perturbation of $D^{\varepsilon w}$, when viewed as acting on the relevant Sobolev space. Hence $\ind(D^{\varepsilon w}) =\ind(D^{\varepsilon w'})$. We find that $\ind(D^{\varepsilon w})$ is independent of $\varepsilon$ and $w$.

By the arguments leading up to~\eqref{eq comp Ag}, with integrals over $s$ replaced by integrals from $t>0$ to $\infty$, we have
\[
A_{g}^t\big(D_C^{\varepsilon w}, \psi\big) = -\frac{1}{2}\eta_g^{\varphi, t}\big(D_N^+ + \varepsilon, \psi\big).
\]
Hence~\eqref{eq index general not inv} becomes
\begin{align}
 \ind_G\big(D^{\varepsilon w}\big)(g) ={}& \LIM_{t\downarrow 0} \bigl(
 \Tr\big(g\circ {\rm e}^{-t\tilde D^-_{\varepsilon w} \tilde D^+_{\varepsilon w}}\psi\big) - \Tr\big(g\circ {\rm e}^{-t\tilde D^+_{\varepsilon w} \tilde D^-_{\varepsilon w}}\psi\big) \bigr) \nonumber
 \\
 &-\frac{1}{2}\eta_g^{\varphi, \reg}\big(D_N^+ + \varepsilon, \psi\big).
\label{eq index non invtble 2}
\end{align}

The coefficients of the heat operator ${\rm e}^{-s \tilde D_{\varepsilon w}^2}$ are continuous in $\varepsilon$. And standard heat kernel asymptotics and localisation apply to ${\rm e}^{-s \tilde D_0^2}$, the analogous operator with $\varepsilon = 0$. These imply that
\[
\lim_{\varepsilon \downarrow 0} \LIM_{t\downarrow 0} \bigl(
 \Tr\big(g\circ {\rm e}^{-t\tilde D^-_{\varepsilon w} \tilde D^+_{\varepsilon w}}\psi\big) - \Tr\big(g\circ {\rm e}^{-t\tilde D^+_{\varepsilon w} \tilde D^-_{\varepsilon w}}\psi\big) \bigr) = \int_{M^g} \psi|_{M^g} \AS_g(D).
\]
Because the left-hand side of~\eqref{eq index non invtble 2} is independent of $\varepsilon$, the claim follows.
\end{proof}

\begin{Remark}The arguments of \cite[Section~4.5]{HW21a} showing that $\psi$ may be replaced by a~step function involve an actual limit $t\downarrow 0$, not the regularised limit $\LIM_{t\downarrow 0}$. For this reason, it is not immediately obvious to us if a version of Theorem~\ref{thm index non invtble} with $\psi$ replaced by a step function is true.
\end{Remark}

\begin{Example}
If $\varphi$ is the zero function, then the left-hand side of~\eqref{eq index non invtble} is the equivariant index of the restriction of $D$ to $M \setminus C$, with Atiyah--Patodi--Singer boundary conditions at $\partial C$. Then for all suitable $\psi$,
a slight modification of the proof of Proposition~\ref{prop cusp contr cylinder} shows that $\eta^{\varphi, \reg}_g\big(D_N^+ +\varepsilon, \psi\big)$ is the regularised $g$-delocalised $\eta$-invariant of $D_N^++\varepsilon$.
Hence
\[
\lim_{\varepsilon \downarrow 0}\eta^{\varphi, \reg}_g\big(D_N^+ +\varepsilon, \psi\big) = \tr\big(g|_{\ker(D_N^+)}\big) + \eta_g^{\reg}\big(D_N^+\big).
\]
This fact is standard; see, for example,~\cite[Lemma~6.7]{HWW}.
\end{Example}

\appendix

\section{Conformal transformations of Dirac operators}\label{app Dirac}

Let $M$ be a manifold of dimension $p$. Let $B_0$ be a Riemannian metric on $M$. Let $S \to M$ be a Clifford module for this metric, with Clifford action $c_0\colon TM \to \End(S)$. Fix a Clifford connection $\nabla^0$ on $S$ preserving a Hermitian metric on $S$, and let $D_0 = c_0\circ \nabla^0$ be the associated Dirac operator. Let $\varphi \in C^{\infty}(M)$, and consider the Riemannian metric $B_{\varphi} := {\rm e}^{2\varphi}B_0$. We denote the gradient operator for $B_0$ by $\grad$.
\begin{Proposition}\label{prop Dirac psi}
There are a Clifford action $c_{\varphi}$ by $TM$ on $S$, with respect to $B_{\varphi}$, and a~Clifford connection $\nabla^{\varphi}$ on $S$, with respect to $c_{\varphi}$ and $B_{\varphi}$, such that the associated Dirac operator $D_{\varphi} = c_{\varphi} \circ \nabla^{\varphi}$ equals
\begin{equation} \label{eq Dpsi}
D_{\varphi} = {\rm e}^{-\varphi}\bigg( D_0 + \frac{p-1}{2} c_0(\grad\varphi) \bigg) = {\rm e}^{-\frac{p+1}{2}\varphi} D_0 {\rm e}^{\frac{p-1}{2}\varphi}.
\end{equation}
\end{Proposition}
\begin{Remark}
The operator $c_0(\grad\varphi)$ in~\eqref{eq Dpsi} is fibrewise antisymmetric. But the operator~$D_{\varphi}$ is symmetric with respect to the $L^2$-inner product defined with the Riemannian density associated to $B_{\varphi}$. This follows, for example, from Proposition~\ref{prop Dirac psi} and the usual argument why Dirac operators are symmetric.
\end{Remark}

We write $\XX(M)$ for the space of smooth vector fields on $M$.
Let $\nabla^{TM, 0}$ be the Levi-Civita connection for $B_0$, and let $\nabla^{TM, \varphi}$ be the Levi-Civita connection for $B_{\varphi}$.
\begin{Lemma}\label{lem LC psi}
For all $v,w \in \XX(M)$,
\[
\nabla^{TM, \varphi}_v w = \nabla^{TM, 0}_v w + v(\varphi) w + w(\varphi)v -B_0(v,w)\grad \varphi.
\]
\end{Lemma}
\begin{proof}
This is a computation based on the
Koszul formulas for $\nabla^{TM, \varphi}$ and $\nabla^{TM, 0}$.
\end{proof}

Consider the Clifford action $c_{\varphi} := {\rm e}^{\varphi}c_0$ with respect to $B_{\varphi}$.
For $A \in \Omega^1(M; \End(S))$, consider the connection $\nabla^A := \nabla^0 + A$ on $S$. For $v \in \XX(M)$, let $A_v \in \End(S)$ be the pairing of $A$ and $v$.
\begin{Lemma}\label{lem B}
The connection $\nabla^A$ is a Clifford connection for $c_{\varphi}$ and $B_{\varphi}$ if and only if for all $v,w \in \XX(M)$,
\[
[A_v, c_0(w)] = w(\varphi) c_0(v) - B_0(v,w) c_0(\grad \varphi).
\]
\end{Lemma}
\begin{proof}
For all $v,w \in \XX(M)$,
\begin{equation} \label{eq B 1}
\big[\nabla^A_v, c_{\varphi}(w)\big] = {\rm e}^{\varphi}[\nabla_v, c_0(w)] + {\rm e}^{\varphi}[A_v, c_0(w)] + {\rm e}^{\varphi}v(\varphi) c_0(w).
\end{equation}
And by Lemma~\ref{lem LC psi},
\begin{equation} \label{eq B 2}
c_{\varphi}\big(\nabla^{TM, \varphi}_v w\big) = {\rm e}^{\varphi}c_0\big(\nabla^{TM, 0}_v w\big) +
{\rm e}^{\varphi} v(\varphi) c_0(w) + {\rm e}^{\varphi}w(\varphi) c_0(v) - {\rm e}^{\varphi} B_0(v,w) c_0(\grad \varphi).
\end{equation}
Because $\nabla^0$ is a Clifford connection for $c_0$ and $B_0$,~\eqref{eq B 1} and~\eqref{eq B 2} are equal if and only if
\begin{equation*}
{\rm e}^{\varphi}[A_v, c_0(w)] + {\rm e}^{\varphi}v(\varphi) c_0(w) = {\rm e}^{\varphi} v(\varphi) c_0(w) + {\rm e}^{\varphi}w(\varphi) c_0(v) - {\rm e}^{\varphi} B_0(v,w) c_0(\grad \varphi).\tag*{\qed}
\end{equation*}
\renewcommand{\qed}{}
\end{proof}

\begin{Lemma}\label{lem cuvw}
For all $u,v,w \in \XX(M)$,
\[
[c_0(u)c_0(v), c_0(w)] = -2B_0(v,w) c_0(u) + 2 B_0(u,w) c_0(v).
\]
\end{Lemma}
\begin{proof}
This is a straightforward computation, involving the equality
\[
c_0(v_1)c_0(v_2) + c_0(v_2) c_0(v_1) = -2B_0(v_1, v_2)
\]
for all $v_1, v_2 \in \XX(M)$.
\end{proof}

Let $f \in C^{\infty}(M)$, and
define $A^{\varphi, f} \in \Omega^1(M; \End(S))$ by
\[
A^{\varphi, f}_v := \frac{1}{2}c_0(\grad \varphi)c_0(v) + fB_0(\grad \varphi, v).
\]
We write $\nabla^{\varphi,f} := \nabla^{A^{\varphi,f}}$.
\begin{Lemma}\label{lem Clifford psi}
For all $f \in C^{\infty}(M)$,
the connection $\nabla^{\varphi, f}$ is a Clifford connection for $c_{\varphi}$ and $B_{\varphi}$.
\end{Lemma}
\begin{proof}
Lemma~\ref{lem cuvw} implies that $A^{\varphi, f}$ satisfies the condition in Lemma~\ref{lem B}.
\end{proof}

\begin{Lemma}\label{lem pres metric}
The connection $\nabla^{\varphi,f}$ preserves the metric on $S$ if and only if $f|_{\supp(\grad \varphi)} = \frac{1}{2}$.
\end{Lemma}
\begin{proof}
Because $\nabla^0$ preserves the metric on $S$, $\nabla^{\varphi,f}$ preserves the same metric if and only if~$A^{\varphi,f}_v$ is anti-Hermitian for any vector field $v$. And because $c_0(w)$ is anti-Hermitian for any vector field~$w$,
\begin{equation*}
\big(A^{\varphi,f}_v\big)^* = -A^{\varphi,f} + (2f-1)B_0(\grad \varphi, v).\tag*{\qed}
\end{equation*}
\renewcommand{\qed}{}
\end{proof}

\begin{proof}[Proof of Proposition~\ref{prop Dirac psi}]
Let $c_{\varphi}$ and $\nabla^{\varphi, \frac{1}{2}}$ as defined above, where $f \equiv 1/2$. Then $\nabla^{\varphi, \frac{1}{2}}$ is a Clifford connection and preserves the metric by Lemmas~\ref{lem Clifford psi} and~\ref{lem pres metric}.

Let $\{e_1, \dots, e_p\}$ be a local orthonormal frame for $TM$ with respect to $B_0$. Then the frame $\big\{{\rm e}^{-\varphi} e_1, \dots, {\rm e}^{-\varphi} e_p\big\}$ is a local orthonormal frame for $TM$ with respect to $B_{\varphi}$. So
\[
D_{\varphi} = \sum_{j=1}^p c_{\varphi}\big({\rm e}^{-\varphi}e_j\big) \nabla^{\varphi, \frac{1}{2}}_{{\rm e}^{-\varphi} e_j}
= {\rm e}^{-\varphi} \sum_{j=1}^p c_0(e_j) \nabla^0_{e_j}
+{\rm e}^{-\varphi} \sum_{j=1}^p c_0(e_j) A^{\varphi, \frac{1}{2}}_{e_j}.
\]
The first term on the right-hand side equals ${\rm e}^{-\varphi} D_0$, and the second term equals
\begin{gather*}
\frac{1}{2}{\rm e}^{-\varphi} \sum_{j=1}^p c_0(e_j) c_0(\grad \varphi) c_0(e_j)
+\frac{1}{2}{\rm e}^{-\varphi} \sum_{j=1}^p c_0(e_j) B_0(\grad \varphi,e_j)\\
 \qquad{} =
\frac{1}{2}{\rm e}^{-\varphi} \sum_{j=1}^p\bigl( c_0(\grad \varphi) -2 B_0(e_j, \grad \varphi)c_0(e_j)\bigr) + \frac{1}{2} c_0(\grad \varphi)\\
\qquad{}
=\frac{p-1}{2}{\rm e}^{-\varphi} c_0(\grad \varphi). \tag*{\qed}
\end{gather*}
\renewcommand{\qed}{}
\end{proof}

\subsection*{Acknowledgements}

We thank Mike Chen for a helpful discussion, and Christian B\"ar for pointing out a useful reference. We are grateful to the referees for several helpful comments and corrections. In particular, we thank the referee who pointed out an error in the previous version of \cite{HW21a}, on which the current paper builds, which has since been fixed.
PH is partially supported by the Australian Research Council, through Discovery Project DP200100729. HW is supported by NSFC-11801178 and Shanghai Rising-Star Program 19QA1403200.

\pdfbookmark[1]{References}{ref}
\LastPageEnding

\end{document}